\documentclass[11pt,final]{article}

\usepackage{amssymb,amsthm,mathtools}
\numberwithin{equation}{section}
\usepackage{bbm}
\usepackage[headings]{fullpage}
\usepackage{graphicx}
\usepackage[
            colorlinks=true,%
            breaklinks=true,%
            linkcolor=blue,
            urlcolor=blue,%
            citecolor=blue,%
            pdfauthor={J. R. Cangelosi and M. Heinkenschloss},%
            bookmarksopen=false,%
            pdfpagemode=None]{hyperref}
\usepackage{times}
\usepackage{tikz}
\usepackage{color}
\usepackage{mathrsfs}
\usepackage[capitalize,nameinlink,noabbrev]{cleveref}
\crefname{assumption}{Assumption}{assumptions}

\newcommand{\bpm}{\begin{pmatrix}}
\newcommand{\epm}{\end{pmatrix}}

\DeclareMathOperator*{\esssup}{ess\,sup}

\newcommand{\eref}[1]{{\rm{(\ref{#1})}}}

\newcount\colveccount
\newcommand*\colvec[1]{
        \global\colveccount#1
        \begin{bmatrix}
        \colvecnext
}
\def\colvecnext#1{
        #1
        \global\advance\colveccount-1
        \ifnum\colveccount>0
                \\
                \expandafter\colvecnext
        \else
                \end{bmatrix}
        \fi
}

\newcommand {\real} {\mathbb{R}}

\newcommand {\cB} {\mathcal{B}}

\newcommand {\cG} {\mathcal{G}}

\newcommand {\cN} {\mathcal{N}}

\newcommand {\cV} {\mathcal{V}}

\newcommand {\cY} {\mathcal{Y}}
\newcommand {\cZ} {\mathcal{Z}}

\newcommand{\BA}{\ensuremath{\mathbf{A}} } %
\newcommand{\BB}{\ensuremath{\mathbf{B}} } %
\newcommand{\BC}{\ensuremath{\mathbf{C}} } %
\newcommand{\BF}{\ensuremath{\mathbf{F}} } %
\newcommand{\BH}{\ensuremath{\mathbf{H}} } %
\newcommand{\BL}{\ensuremath{\mathbf{L}} } %
\newcommand{\BQ}{\ensuremath{\mathbf{Q}} } %
\newcommand{\BR}{\ensuremath{\mathbf{R}} } %
\newcommand{\bc}{\ensuremath{\mathbf{c}}} %
\newcommand{\bd}{\ensuremath{\mathbf{d}}} %
\newcommand{\bff}{\ensuremath{\mathbf{f}}} %
\newcommand{\bg}{\ensuremath{\mathbf{g}}} %
\newcommand{\bp}{\ensuremath{\mathbf{p}}} %
\newcommand{\bq}{\ensuremath{\mathbf{q}}} %
\newcommand{\br}{\ensuremath{\mathbf{r}}} %
\newcommand{\bu}{\ensuremath{\mathbf{u}}} %
\newcommand{\bv}{\ensuremath{\mathbf{v}}} %
\newcommand{\bx}{\ensuremath{\mathbf{x}}} %
\newcommand{\by}{\ensuremath{\mathbf{y}}} %
\newcommand{\bz}{\ensuremath{\mathbf{z}}} %
\newcommand{\balpha}{\ensuremath{\mbox{\boldmath $\alpha$}}}

\newcommand {\bgamma} {\mbox{\boldmath $\gamma$}}
\newcommand {\blambda} {\mbox{\boldmath $\lambda$}}

\newcommand {\bdelta} {\boldsymbol{\delta}}

\newcommand {\bepsilon} {\mbox{\boldmath $\epsilon$}}

\newcommand {\bpsi} {\boldsymbol{\psi}}

\newcommand {\BPhi} {\mbox{\boldmath $\Phi$}}

\newcommand {\bphi} {\mbox{\boldmath $\phi$}}

\newcommand{\aall}{\mathrm{a.a.} \;}

\newcommand{\WW}{\big( W^{1, \infty}(I) \big)}

\newcommand{\LL}{\big(L^\infty(I) \big)}

\newtheorem{theorem}{Theorem}[section]
\newtheorem{lemma}[theorem]{Lemma}

\newtheorem{remark}[theorem]{Remark}

\newtheorem{assumption}[theorem]{Assumption}

               {\refstepcounter{theorem}\vspace{2ex}%
                \noindent{\bf Examples \thetheorem} }%
               {\hfill $\diamond$}

               {\refstepcounter{theorem}%
                \noindent{\bf Example \thetheorem} }%
               {\hfill $\diamond$}

\begin{document}
\title{{\Large \bf Sensitivity of Optimal Control Solutions and Quantities of Interest with Respect to Component Functions}
         \thanks{This research was supported in part by AFOSR Grant FA9550-22-1-0004 NSF Grant DMS-2231482. }
         }
\author{Jonathan R. Cangelosi
             \thanks{Department of Computational Applied Mathematics and Operations Research,
                      MS-134, Rice University, 6100 Main Street,
                     Houston, TX 77005-1892. 
                     E-mail: jrc20@rice.edu}
             \and
             Matthias Heinkenschloss
             \thanks{Department of Computational Applied Mathematics and Operations Research,
                 MS-134, Rice University, 6100 Main Street,
                 Houston, TX 77005-1892, and the Ken Kennedy Institute, Rice University.
                 E-mail: heinken@rice.edu}
        }

\date{\today}

\maketitle

\begin{abstract}
This work establishes sensitivities of the solution of an optimal control problem (OCP) and a corresponding quantity of interest (QoI) to perturbations in a state/control-dependent component function that appears in the governing ODEs and the objective function. This extends existing OCP sensitivity results, which consider the sensitivity of the OCP solution with respect to state/control-independent parameters.
 It is shown that with Carath\'eodory-type assumptions, the Implicit Function Theorem can be applied to establish continuous 
 Fr\'echet differentiability of the OCP solution with respect to the component function. 
 These sensitivities are used to develop new estimates for the change in the QoI when the component function is perturbed. 
Applicability of the theoretical results is demonstrated on a trajectory optimization problem for a hypersonic vehicle.
\end{abstract}



\noindent

\section{Introduction}   \label{sec:intro}
We establish results  for the sensitivity of the solution of an optimal control problem (OCP) governed by ordinary
differential equations (ODEs) and of an OCP solution-dependent quantity of interest (QoI) 
with respect to perturbations in a state/control-dependent component function that appears in the governing 
ODEs or in the objective function. 
Our results extend existing OCP sensitivity results that consider the sensitivity of the OCP solution with respect to 
state/control-independent parameters, such as those in 
 \cite{DWPeterson_1974a},
\cite{AVFiacco_1983},
 \cite{ALDontchev_1983},
 \cite{HMaurer_HJPesch_1994a},
 \cite{KMalanowski_1995a},
 \cite{JFBonnans_AShapiro_2000},
 \cite{RGriesse_BVexler_2007a},
 \cite[Sec.~6.1]{MGerdts_2012a}.
Our new sensitivity results are useful in many applications where a physical system is modeled by ODEs involving component functions 
or constitutive laws that are only known experimentally at a few points or must be computed through expensive, high-fidelity computations. 
In such cases, these expensive-to-evaluate component functions are approximated by surrogate models, and these surrogate models are used
in the solution of the OCP.  As a specific example, we consider trajectory optimization for a hypersonic glide vehicle where state- and control-dependent lift, drag, and moment coefficients must ideally be computed from expensive computational fluid dynamics (CFD) computations and are approximated by surrogates constructed from a few CFD computations at select points in state/control space.
Our sensitivity results provide tools to assess the impact of surrogate errors on 
the OCP solution or a solution-dependent QoI.
To establish our sensitivity analysis, we use a  (somewhat nonstandard) superposition or Nemytskii operator to express the optimality
conditions as an operator equation in suitable Banach spaces, and then we apply the Implicit Function Theorem to establish Fr\'echet 
differentiability of the mapping from component functions to OCP solutions. 
Establishing the hypotheses of the Implicit Function Theorem is more difficult than in the parametric case, as it requires 
proving the continuous Fr\'echet differentiability of the Nemytskii operator.

We also combine our new sensitivity results with adjoint-based analysis 
to develop sensitivity-based error  estimates for a solution-dependent QoI when a pointwise bound for the error between the approximate component function used in 
the OCP and ``true'' component function is available. These new sensitivity-based error estimates could be used to refine surrogates of the component functions,
similar to what has recently been proposed in the ODE context \cite{JRCangelosi_MHeinkenschloss_2024a}, or similar to sensitivity-driven data acquisition approaches used in the OCP
context heuristically without such error bounds  \cite{JRCangelosi_MHeinkenschloss_JTNeedels_JJAlonso_2024a}, \cite{JTNeedels_JJAlonso_2024a},
or for parametrized models \cite{JHart_BvanBloemenWaanders_LHood_JParish_2023a}, \cite{JHart_BvanBloemenWaanders_2025a}.
Like traditional OCP sensitivity analysis, our new sensitivity analysis could also be applied in model predictive control
\cite{CBueskens_HMaurer_2000a}, \cite{LGruene_JPannek_2011a}, \cite{JBRawlings_DQMayne_MMDiehl_2024a}. 
Other existing work in this direction has primarily focused on leveraging 
statistical sensitivity information such as Sobol indices \cite{IMSobol_2001a},  \cite{IMSobol_SKucherenko_2009a}, 
\cite{MVohra_AAlexanderian_CSafta_SMahadevan_2019a} 
which, like   \cite{JHart_BvanBloemenWaanders_LHood_JParish_2023a}, \cite{JHart_BvanBloemenWaanders_2025a},
is done in discrete settings rather than the infinite-dimensional settings in which many dynamical systems are embedded.

Given an interval $I := (t_0, t_f)$, an initial state $x_0 \in \real^{n_x}$, and functions
\begin{align*}
	\varphi &: \real^{n_x} \times \real^{n_p} \rightarrow \real, &
	l &: I \times \real^{n_x} \times \real^{n_u} \times \real^{n_p} \times \real^{n_g} \rightarrow \real, \\
	\bff &: I \times \real^{n_x} \times \real^{n_u} \times \real^{n_p} \times \real^{n_g} \rightarrow \real^{n_x}, &
	\bg &: I \times \real^{n_x} \times \real^{n_u} \times \real^{n_p} \rightarrow \real^{n_g},
\end{align*}
we will investigate the dependence of the solution
\begin{align} \label{eq:intro:optimization-variables}
	\bx : \overline{I} \rightarrow \real^{n_x}, && \bu : I \rightarrow \real^{n_u}, && \bp \in \real^{n_p}
\end{align}
to the Bolza-form optimal control problem (OCP)
\begin{subequations} \label{eq:intro:OCP}
\begin{align}
	\min_{\bx, \bu, \bp} \quad & \varphi\big(\bx(t_f), \bp \big) + \int_{t_0}^{t_f} l\Big(t, \bx(t), \bu(t), \bp, \bg \big( t, \bx(t), \bu(t), \bp \big) \Big) \, dt \\
	\mbox{s.t.} \quad 
	& \bx'(t) = \bff \Big(t, \bx(t), \bu(t), \bp, \, \bg \big(t, \bx(t), \bu(t), \bp \big) \Big), \qquad \mbox{ almost all (a.a.) } t \in I, \\
	& \bx(t_0) = x_0,
\end{align}
\end{subequations}
on the component function $\bg$ appearing in the dynamics and the objective function.
The precise function space setting will be specified in \cref{sec:sensitivity}.
The state is denoted by $\bx$, the control is denoted by $\bu$, and parameters are denoted by $\bp$.
We often use $\bx(\cdot \, ; \bg)$, $\bu(\cdot \, ; \bg)$, $\bp(\bg)$ to emphasize that the solution of \eqref{eq:intro:OCP}  is computed with the component function $\bg$. 

In \cref{sec:sensitivity} we will specify the function space setting for \eqref{eq:intro:OCP} and prove continuous Fr\'echet differentiability of the solution mapping
\begin{equation} \label{eq:intro:solution-mapping}
	\bg \mapsto \bz(\bg) := \big( \bx(\cdot \, ; \bg), \bu(\cdot \, ; \bg), \bp(\bg), \blambda(\cdot \, ; \bg) \big)
\end{equation}
(where $\blambda$ denotes the costate),
and also the continuous Fr\'echet differentiability of a QoI 
\begin{subequations} \label{eq:intro:QoI}
\begin{equation}
	\bg \mapsto \widetilde{q}(\bg) := q\big( \bx(\cdot \, ; \bg), \bu(\cdot \, ; \bg), \bp(\bg); \bg \big),
\end{equation}
where
\begin{equation} \label{eq:intro:QoI-bolza}
	q(\bx, \bu, \bp; \bg) := \phi\big( \bx(t_f), \bp \big) + \int_{t_0}^{t_f} \ell\Big(t, \bx(t), \bu(t), \bp, \, \bg \big(t, \bx(t), \bu(t), \bp \big) \Big) \, dt,
\end{equation}
\end{subequations}
and $\bx(\cdot \, ; \bg), \bu(\cdot \, ; \bg), \bp(\bg)$ is the solution of \eqref{eq:intro:OCP}.
The QoI \eqref{eq:intro:QoI} has the same mathematical form as the objective function in \eqref{eq:intro:OCP} but may be a different quantity entirely.
 
In applications, one often desires the solution of the OCP \eqref{eq:intro:OCP} at a true component function $\bg _*$ which is not available analytically.
Instead, one constructs a computationally inexpensive approximation $\widehat{\bg}$ and solves the OCP \eqref{eq:intro:OCP} with $\bg = \widehat{\bg}$ to obtain an approximate solution.
In \cref{sec:QoI_error}, we will apply the sensitivity results of \cref{sec:sensitivity} to derive an approximate upper bound for the QoI error 
$| \widetilde{q}(\widehat{\bg}) - \widetilde{q}(\bg_*) |$ given a pointwise bound for $| \widehat{\bg} - \bg_* |$ (understood componentwise). 
This sensitivity-based error bound gives a way to assess whether the approximation $\widehat{\bg} \approx \bg_*$ is good enough to obtain a meaningful solution to \eqref{eq:intro:OCP}.
As mentioned earlier, this sensitivity-based error bound can be used to systematically refine the surrogate $\widehat{\bg}$, which will be explored in future work.

In \cref{sec:numerics}, the results of \cref{sec:sensitivity} and~\cref{sec:QoI_error} will be applied to a numerical example involving trajectory optimization for a notional hypersonic vehicle with aerodynamic coefficient models approximated by surrogates.

{\bf Notation.} 
We will use $\| \cdot \|$ to denote a vector norm on $\real^m$ (where $m$ depends on the context) or an 
induced matrix norm.
By $\mathcal{B}_R(0) \subset \real^m$ we denote the closed ball in $\real^m$ around zero with radius $R>0$.
When infinite-dimensional normed linear spaces are considered, the norm will always be specified explicitly 
using subscripts.

Given an interval $I = (t_0, t_f)$,  $\big( L^\infty(I) \big)^m$ denotes the Lebesgue space of essentially bounded functions on $I$ with
values in $\real^m$, and $\big( W^{1, \infty}(I) \big)^m$ denotes the Sobolev space of functions on $I$ with 
values in $\real^m$ that are weakly differentiable on $I$ and have essentially bounded derivative.

We typically use bold font for vector- or matrix-valued functions and regular font for scalars, vectors, matrices, and scalar-valued functions (except states, which will be boldface). For example, the function $\bx: I \rightarrow \real^{n_x}$ has values $\bx(t) \in \real^{n_x}$, and $x \in \real^{n_x}$ denotes a vector. This distinction will be useful when
studying compositions of functions. Also, when using subscripts for derivatives, regular subscripts will be used to denote partial derivatives with respect to a vector, while boldface subscripts will be used to denote Fr\'echet derivatives with respect to a function.


\section{Sensitivity Analysis} \label{sec:sensitivity}
To compute the sensitivities of a local solution of the OCP \eqref{eq:intro:OCP} we apply the
Implicit Function Theorem in Banach spaces  to the system of necessary optimality conditions.
In contrast to standard parametric sensitivity analyses, the presence of the state/control/parameter-dependent component function $\bg$ requires the use of a superposition or Nemytskii operator.
We review the problem setting for the state equation (\ref{eq:intro:OCP}b,c)
in \cref{sec:sensitivity-problem-setting}.
In \cref{sec:sensitivity-differentiability-dynamics} we establish the continuous Fr\'echet differentiability  
of Nemytskii operators associated with the right-hand side function in (\ref{eq:intro:OCP}b,c) and 
with the integrand in the objective  (\ref{eq:intro:OCP}a) or the  quantity of interest \eqref{eq:intro:QoI-bolza}.
The sensitivity of the solution of \eqref{eq:intro:OCP} with respect to $\bg$ is established in 
\cref{sec:sensitivity-OCP}.
Finally, we will use the adjoint equation approach to efficiently compute the sensitivity of
the QoI (\ref{eq:intro:QoI}a) in \cref{sec:sensitivity-QoI}.

For completeness, we state the Implicit Function Theorem in Banach spaces.
See, e.g., \cite[Thm.~15.1]{KDeimling_1985}, \cite[Thm.~2.1.14]{MGerdts_2012a},
\cite[Thm.~4.E, p.~250]{EZeidler_1995b}.
\begin{theorem}[Implicit Function Theorem] \label{thm:Sensitivity:intro:IFT}
	Let $\cZ, \cG, \cV$ be Banach spaces, let $D \subset \cZ \times \cG$ be a neighborhood of the point $(\overline{\bz}, \overline{\bg}) \in \cZ \times \cG$, and let 
	$\bpsi : D \rightarrow \cV$ 
	be an operator satisfying $\bpsi(\overline{\bz}, \overline{\bg}) = 0$. If 
	\begin{enumerate}
		\item[(i)] $\bpsi$ is continuously Fr\'echet differentiable and
		
		\item[(ii)] the partial Fr\'echet derivative $\bpsi_\bz(\overline{\bz}, \overline{\bg})$ is bijective,
	\end{enumerate}
	 then there exist neighborhoods $\cN(\overline{\bz}) \subset \cZ$, $\cN(\overline{\bg}) \subset \cG$ and a unique mapping 
	 $\bz : \cN(\overline{\bg}) \rightarrow \cN(\overline{\bz})$ that  is continuously Fr\'echet differentiable and satisfies
	$\bz(\overline{\bg}) = \overline{\bz}$ and
	$\bpsi\big(\bz(\bg), \bg \big) = 0$ for all  $\big(\bz(\bg), \bg \big) \in \cN(\overline{\bz}) \times \cN(\overline{\bg})$.
	The Fr\'echet derivative is
	\begin{equation}  \label{eq:Sensitivity:intro:sensitivity_eqn}
		\bz_\bg(\overline{\bg}) = -\bpsi_\bz(\overline{\bz}, \overline{\bg})^{-1} \, \bpsi_\bg(\overline{\bz}, \overline{\bg}).
	\end{equation}
\end{theorem}

\subsection{Problem Setting}   \label{sec:sensitivity-problem-setting}
We begin with the state equation (\ref{eq:intro:OCP}b,c).
To allow for control inputs $\bu$ that are discontinuous in time (e.g., bang-bang controls), we consider the state equation
 in the sense of Carath\'eodory. 
 First, consider the  initial value problem 
\begin{equation} \label{eq:Sensitivity:caratheodory:IVP-Caratheodory}
\begin{aligned}
	\bx'(t) &= \widetilde{\bff} \big( t, \bx(t) \big), \qquad \aall t \in I, \\
	\bx(t_0) &= x_0.
\end{aligned}
\end{equation}
The following conditions are used to establish existence and uniqueness of solutions to \eqref{eq:Sensitivity:caratheodory:IVP-Caratheodory} in suitable function spaces.
\begin{assumption} \label{as:Sensitivity:caratheodory:caratheodory-conditions}
	Let the function $\widetilde{\bff} : I \times \real^{n_x} \rightarrow \real^{n_x}$ satisfy the following:
	\begin{itemize}
		\item[(i)] $\widetilde{\bff}$ is measurable in $t \in I$ for each $x \in \real^{n_x}$, it is continuous in $x \in \real^{n_x}$ for almost all $t \in I$, and there exists an integrable function $m : I \rightarrow \real$ such that
		\[
			\| \widetilde{\bff}(t, x) \| \leq m(t), \qquad \aall t \in I \mbox{ and all } x \in \real^{n_x};
		\]
		
		\item[(ii)] There exists an integrable function $l : I \rightarrow \real$ such that
	\[
		\| \widetilde{\bff}(t, x_1) - \widetilde{\bff}(t, x_2) \| \leq l(t) \| x_1 - x_2 \|, \qquad \aall t \in I \mbox{ and all } x_1, x_2 \in \real^{n_x};
	\]
	
		\item[(iii)] The function $m$ in condition (i) belongs to $L^\infty(I)$.
	\end{itemize}
\end{assumption}

\begin{remark} \label{rk:Sensitivity:caratheodory:bounded-set}
     \cref{as:Sensitivity:caratheodory:caratheodory-conditions} may seem restrictive given that $\real^{n_x}$ is unbounded. 
      However, it is typical in applications to only encounter states within a (sufficiently large) ball $\cB_R(0) \subset \real^{n_x}$, 
      and in such cases it suffices that the restriction of $\widetilde{\bff}$ to $I \times \cB_R(0)$ admit an extension to 
      $I \times \real^{n_x}$ with the required properties.
\end{remark}

\begin{theorem} \label{thm:Sensitivity:caratheodory:caratheodory}
	If \cref{as:Sensitivity:caratheodory:caratheodory-conditions}~(i) holds, then
     the IVP \eqref{eq:Sensitivity:caratheodory:IVP-Caratheodory} has a solution $\bx \in \big( W^{1, 1}(I) \big)^{n_x}$. 
     If \cref{as:Sensitivity:caratheodory:caratheodory-conditions}~(i)-(ii) holds, then
     the IVP \eqref{eq:Sensitivity:caratheodory:IVP-Caratheodory} has a unique solution $\bx \in \big( W^{1, 1}(I) \big)^{n_x}$. 
     If \cref{as:Sensitivity:caratheodory:caratheodory-conditions} (i)-(iii) holds, then
     the IVP \eqref{eq:Sensitivity:caratheodory:IVP-Caratheodory} has a unique solution $\bx \in \WW^{n_x}$.
\end{theorem}
\begin{proof}
      The first two parts of the theorem follow from \cite[Thm.~1-2]{AFFilippov_1988a}. 
      The last part simply follows from the fact that $\big\| \widetilde{\bff}\big(t, \bx(t) \big) \big\| \leq m(t)$, $\aall t \in I$,
      where $m \in L^\infty(I)$, so $\bx'$ is essentially bounded, and thus $\bx \in \WW^{n_x}$.
\end{proof}

With appropriate assumptions on $\bff$ and $\bg$, \cref{thm:Sensitivity:caratheodory:caratheodory} may be applied to establish existence and uniqueness of solutions for  (\ref{eq:intro:OCP}b,c) by taking
\begin{equation} \label{eq:Sensitivity:caratheodory:tildef=f}
	\widetilde{\bff}(t, x) := \bff \big( t, x, \bu(t), \bp, \bg(t, x, \bu(t), \bp ) \big),
\end{equation}
where $\bu \in \LL^{n_u}$ and $\bp \in \real^{n_p}$ are given. This motivates the identification of appropriate assumptions for $\bff$ and the selection of a suitable Banach space for the functions $\bg$ appearing in (\ref{eq:intro:OCP}b,c).

In what follows the notation $y = (x, u, p)$ will be used to denote state-control-parameter triples in
\[
	\real^{n_y} := \real^{n_x} \times \real^{n_u} \times \real^{n_p}.
\]
The set of model functions $\bg$ appearing in (\ref{eq:intro:OCP}b,c) will be given by the Banach space
\begin{subequations} \label{eq:Sensitivity:caratheodory:g-space}
\begin{align}
	\big(\cG^s(I)\big)^{n_g} := \big\{ \bg &: I \times \real^{n_y} \rightarrow \real^{n_g} \; : \; 
	                    \bg( t, y ) \mbox{ is $s$-times continuously partially }   \nonumber \\
	                   & \mbox{differentiable  with respect to } y \in \real^{n_y}  \mbox{ for } \aall t \in I,   \nonumber \\
                           &   \mbox{is measurable in $t$ for each $y \in \real^{n_y}$, and }  
                             \| \bg \|_{(\cG^s(I))^{n_g}}  < \infty \big\},
\end{align}
where
\begin{equation} \label{eq:Sensitivity:caratheodory:g-space-norm}
	\| \bg \|_{(\cG^s(I))^{n_g}} := \sum_{n=0}^s \; \esssup_{t \in I} \; \sup_{y \in \real^{n_y}} \left\| \frac{\partial^n}{\partial y^n} \bg(t, y) \right\|
\end{equation}
\end{subequations}
for $s = 3$. The $s = 2$ case will also be useful in the sensitivity analysis. Note that $\big(\cG^3(I)\big)^{n_g}$ is continuously embedded in $\big(\cG^2(I)\big)^{n_g}$, so all differentiability results with respect to $\bg \in \big(\cG^2(I)\big)^{n_g}$ will also hold for $\bg \in \big(\cG^3(I)\big)^{n_g}$.
\begin{remark} \label{rk:Sensitivity:caratheodory:bounded-set-g}
	The assumptions on $\bg \in \big(\cG^s(I)\big)^{n_g}$ may seem strong at first; 
	however, just like in Remark~\ref{rk:Sensitivity:caratheodory:bounded-set}, it often suffices in 
	applications for $\bg(t, y)$ to be specified for almost all $t \in I$ and all points $y$ within some 
	sufficiently large ball $\cB_R(0) \subset \real^{n_y}$ and admit an extension to the whole space with the required properties.
\end{remark}

In the following, the notation $\bg_y(t, y) \in \real^{n_g \times n_y}$ denotes the partial Jacobian of $\bg$ with respect to $y$ 
at $(t, y) \in I \times \real^{n_y}$,
and $\bg_{yy}(t, y) \in \real^{n_g \times n_y \times n_y}$ denotes the tensor of partial Hessians of the components of $\bg$ 
with respect to $y$ at $(t, y) \in I \times \real^{n_y}$.
Three-times continuous differentiability of $\bg \in \big(\cG^3(I)\big)^{n_g}$ is needed to establish sensitivities for the solution of
the optimal control problem \eqref{eq:intro:OCP}. Fewer assumptions are needed to establish the existence and uniqueness of 
a solution of the state equation (\ref{eq:intro:OCP}b,c), or the characterization of the solution of the optimal control problem \eqref{eq:intro:OCP}.

The following assumption on $\bff$ ensures the existence and uniqueness of solutions to (\ref{eq:intro:OCP}b,c) in $\WW^{n_x}$.
\begin{assumption} \label{as:Sensitivity:caratheodory:IVP_unique_solution}
	Let the function $\bff : I \times \real^{n_y} \times \real^{n_g} \rightarrow \real^{n_x}$ satisfy the following:
   \begin{itemize}
   \item[(i)] $\bff$ is measurable in $t \in I$ for each $y \in \real^{n_y}$ and $g  \in \real^{n_g}$, 
   	it is continuous in $y \in \real^{n_y}$ and $g \in \real^{n_g}$ for almost all $t \in I$,
      and there exists an integrable function $m_f : I \rightarrow \real$ such that $$\| \bff( t, y, g ) \| \le m_f(t) (1 + \| g \|), \qquad \aall t \in I \mbox{ and all } y \in \real^{n_y}, g \in \real^{n_g}; $$ 
      
   \item[(ii)]   There exists an integrable function $l_f : I \rightarrow \real$ such that 
   \begin{align*}
   \| \bff( t, y_1, g_1 ) -  \bff( t, y_2, g_2 )  \| &\le l_f(t) ( \| y_1 - y_2  \| + \| g_1 - g_2  \| ), \\ &  \qquad \aall t \in I \mbox{ and all } y_1, y_2 \in \real^{n_y},  g_1, g_2 \in \real^{n_g}.
   \end{align*}
   \item[(iii)] The function $m_f$ in condition (i) belongs to $L^\infty(I)$.
\end{itemize}
\end{assumption}

\begin{theorem} \label{thm:Sensitivity:caratheodory:IVP_unique_solution}
     If \cref{as:Sensitivity:caratheodory:IVP_unique_solution}~(i) holds, then given any $\bg \in \big(\cG^1(I)\big)^{n_g}$
     the IVP (\ref{eq:intro:OCP}b,c) has a solution $\bx(\cdot \, ; \bg) \in \big( W^{1, 1}(I) \big)^{n_x}$. 
     If \cref{as:Sensitivity:caratheodory:IVP_unique_solution}~(i)-(ii) holds, then given any $\bg \in \big(\cG^1(I)\big)^{n_g}$
     the IVP (\ref{eq:intro:OCP}b,c) has a unique solution $\bx(\cdot \, ; \bg) \in \big( W^{1, 1}(I) \big)^{n_x}$. 
     If \cref{as:Sensitivity:caratheodory:IVP_unique_solution} (i)-(iii) holds, then given any $\bg \in \big(\cG^1(I)\big)^{n_g}$
     the IVP (\ref{eq:intro:OCP}b,c) has a unique solution $\bx(\cdot \, ; \bg) \in \WW^{n_x}$.
\end{theorem}
\begin{proof}
   Taking \eqref{eq:Sensitivity:caratheodory:tildef=f}, we proceed similarly to \cref{thm:Sensitivity:caratheodory:caratheodory}.

      Let \cref{as:Sensitivity:caratheodory:IVP_unique_solution} (i) hold. 
      Since $\bff\big( t, y, \bg(t, y) \big)$ is measurable in its first argument and continuous in its third argument,
      which itself is measurable in $t$, $\widetilde{\bff}$ is measurable in $t$. 
      Similarly, $\bff\big( t, y, \bg(t, y) \big)$ is continuous in its second and third arguments, the latter of which 
      is itself continuous in $y$, so $\widetilde{\bff}$ is continuous in $y$. Finally, we have 
    \begin{equation} \label{eq:Sensitivity:caratheodory:nemytskii-L-infty}
    	\big\| \bff \big( t, y, \bg(t, y) \big) \big\| \leq m_f(t) (1 + \| \bg(t, y) \| ) \leq m_f(t) (1 + \| \bg \|_{(\cG^1(I))^{n_g}})
    \end{equation}
    for $\aall t \in I$ and all $y \in \real^{n_y}$,  so
    $m(t) := m_f(t) ( 1 + \| \bg \|_{(\cG^1(I))^{n_g}} )$
     satisfies \cref{as:Sensitivity:caratheodory:caratheodory-conditions} (i). 
     Thus, the first part of the theorem follows from \cref{thm:Sensitivity:caratheodory:caratheodory}.

     Next, let \cref{as:Sensitivity:caratheodory:IVP_unique_solution} (i)-(ii) hold, and observe that
    \begin{align*}
    	\big\| \bff \big( t, y_1, \bg(t, y_1) \big) - \bff \big( t, y_2, \bg(t, y_2) \big) \big\| 
	&\leq l_f(t) (\| y_1 - y_2 \| + \| \bg(t, y_1) - \bg(t, y_2) \|) \\ &\leq l_f(t) (\| y_1 - y_2 \| + \| \bg \|_{(\cG^1(I))^{n_g}} \| y_1 - y_2 \|) 
    \end{align*}
    for $\aall t \in I$ and all $y_1, y_2 \in \real^{n_y}$.
    This implies that $l(t) := l_f(t) (1 + \| \bg \|_{(\cG^1(I))^{n_g}})$ satisfies \cref{as:Sensitivity:caratheodory:IVP_unique_solution} (ii), 
     so the second part of the theorem follows from \cref{thm:Sensitivity:caratheodory:caratheodory}.

      Finally, the last part of the theorem follows from \cref{thm:Sensitivity:caratheodory:caratheodory} since 
      $m$ as defined previously is a constant times $m_f \in L^\infty(I)$.
\end{proof}

\subsection{Fr\'echet Differentiability of the Dynamics}    \label{sec:sensitivity-differentiability-dynamics}
We consider the Fr\'echet differentiability  of a map associated with the right-hand side function in (\ref{eq:intro:OCP}b,c) and of
a map associated  with a scalar-valued function
\[
           \ell : I \times \real^{n_x} \times \real^{n_u} \times \real^{n_p} \times \real^{n_g} \rightarrow \real.
\]
This function could be the integral of the objective  (\ref{eq:intro:OCP}a), i.e., $\ell = l$, or it could be associated with
a quantity of interest \eqref{eq:intro:QoI-bolza}.

In the following, we will use 
\[
	\by(t) := \big( \bx(t), \bu(t), \bp \big), \qquad \aall t \in I
\]
and the space
\[
	\cY^\infty := \WW^{n_x} \times \LL^{n_u} \times \real^{n_p}
\]
for the optimization variables $(\bx, \bu, \bp)$ outlined in \eqref{eq:intro:optimization-variables}.

Let $s \ge 1$ and let \cref{as:Sensitivity:caratheodory:IVP_unique_solution} (i)-(iii) hold with a corresponding assumption for $\ell$.
Define the Nemytskii (superposition) operators
\begin{subequations}   \label{eq:Sensitivity:nemytskii:nemytskii-FL}
\begin{align}
 \label{eq:Sensitivity:nemytskii:nemytskii-F}
   \cY^\infty \times \big(\cG^s(I)\big)^{n_g} \ni (\by, \bg) \mapsto \BF(\by, \bg) 
    &:= \bff \Big( \cdot, \by(\cdot), \bg \big( \cdot, \by(\cdot) \big) \Big) \in \LL^{n_x}, \\
  \label{eq:Sensitivity:nemytskii:nemytskii-L}
    \cY^\infty \times \big(\cG^s(I)\big)^{n_g} \ni (\by, \bg) \mapsto \BL(\by, \bg) 
    &:= \ell \Big( \cdot, \by(\cdot), \bg \big( \cdot, \by(\cdot) \big) \Big) \in L^\infty(I).
\end{align}
\end{subequations}
Note that under \cref{as:Sensitivity:caratheodory:IVP_unique_solution} (i)-(iii)
the outputs of the operators are essentially bounded.

To conduct sensitivity analysis on solutions of \eqref{eq:intro:OCP} and the resulting QoIs \eqref{eq:intro:QoI}, one 
requires continuous Fr\'echet differentiability of $\BF$ and $\BL$. The following theorem establishes 
continuous Fr\'echet differentiability for a class of Nemytskii operators that include \eqref{eq:Sensitivity:nemytskii:nemytskii-FL},
generalizing results on (continuous) Fr\'echet differentiability of Nemytskii operators in $L^\infty$ spaces such as \cite[Lem.~4.12,~4.13]{FTroeltzsch_2010a} to our setting.
\begin{theorem} \label{thm:Sensitivity:nemytskii:nemytskii-derivative}
	Let $\bphi : I \times \real^{n_y} \times \real^{n_g} \rightarrow \real^{n_\phi}$ be a given function such that 
	$\bphi(t, y, g)$ is measurable in $t \in I$ for all $y \in \real^{n_y}$, $g \in \real^{n_g}$,  is continuously 
	partially differentiable with respect to $y \in \real^{n_y}$ and $g  \in \real^{n_g}$ for a.a.\  $t \in I$,
	and satisfies the following conditions:
   \begin{itemize}                   
   \item[(i)] \emph{Boundedness:}  There exists $K$ such that $ \| \bphi_y( t, 0, 0 )  \| \le K$ and 
                    $ \| \bphi_g( t, 0, 0 )  \| \le K$  for  a.a.\ $t \in I$;
                   
   \item[(ii)]  \emph{Local Lipschitz continuity:}  For all $R > 0$ there exists $L(R)$ such that
    \begin{align*}   
    	&\| \bphi_x(t, y_1, g_1) - \bphi_x(t, y_2, g_2) \| + \| \bphi_g(t, y_1, g_1) - \bphi_g(t, y_2, g_2) \| \\
          &\qquad \leq L(R) ( \| y_1 - y_2 \| + \| g_1 - g_2 \| ), \\
          &\qquad\qquad\qquad\qquad\qquad\qquad  \aall  t \in I  \mbox{ and all } y_1, y_2 \in \cB_R(0), \; g_1, g_2 \in \cB_R(0).
    \end{align*}
    \end{itemize}
	Then the Nemytskii operator $\BPhi : \LL^{n_y} \times \big(\cG^2(I)\big)^{n_g} \rightarrow \LL^{n_\phi}$ given by
	\begin{equation} \label{eq:Sensitivity:nemytskii:nemytskii-operator}
		\BPhi(\by, \bg) := \bphi \Big( \cdot, \by(\cdot), \bg \big( \cdot, \by(\cdot) \big) \Big)
	\end{equation}
	is continuously Fr\'echet differentiable, and its derivative is given by
	\begin{equation} \label{eq:Sensitivity:nemytskii:Frechet_deriv}
	\begin{aligned}
      	&[\BPhi'(\by, \bg) (\delta \by, \delta \bg)](t)   \\
	&= \Big[ \bphi_y\Big( t, \by(t), \bg\big(t, \by(t) \big) \Big) 
	              + \bphi_g\Big( t, \by(t), \bg\big(t, \by(t) \big) \Big) \bg_y\big(t, \by(t) \big) \Big] \delta \by(t)    \\ 
	 &\qquad + \bphi_g\Big( t, \by(t), \bg\big( t,  \by(t) \big) \Big) \delta \bg\big(t,  \by(t) \big).
    \end{aligned}
    \end{equation}
\end{theorem}
\begin{proof}
      This result is a slight but straightforward extension of  \cite[Thm.~2.5]{JRCangelosi_MHeinkenschloss_2024a}.
\end{proof}

\begin{remark} \label{rk:sensitivity:cts-embed}
     In \cref{thm:Sensitivity:nemytskii:nemytskii-derivative}, the space
     $\big(\cG^2(I)\big)^{n_g}$ is needed for continuous Fr\'echet differentiability of $\BPhi$;
     see \cite[Thm.~2.6]{JRCangelosi_MHeinkenschloss_2024a}.
      The space $\big(\cG^1(I)\big)^{n_g}$ for the component function $\bg$ is sufficient for Fr\'echet 
      differentiability of $\BPhi$ at certain points, but is insufficient for continuous differentiability;
      see \cite[Thm.~2.4]{JRCangelosi_MHeinkenschloss_2024a}.
	
     \cref{thm:Sensitivity:nemytskii:nemytskii-derivative} considers differentiability in $\LL^{n_y}$ rather than $\cY^\infty$.
     Because $\cY^\infty$ is continuously embedded in $\LL^{n_y}$, the result applies to 
     the Nemytskii operator  \eqref{eq:Sensitivity:nemytskii:nemytskii-FL} with $s\ge 2$.
\end{remark}

\subsection{Fr\'echet Differentiability of the OCP Solution}    \label{sec:sensitivity-OCP}
The sensitivity analysis for \eqref{eq:intro:OCP} is based on the first-order necessary optimality conditions for \eqref{eq:intro:OCP}.
See, e.g., \cite[Sec.~3.1]{MGerdts_2012a}, \cite[Sec.~5.2]{JJahn_2007a} for a derivation of these conditions for general optimal control problems. 
The following smoothness assumption is needed to obtain the optimality conditions and is adapted from \cite[Assumption~2.2.8]{MGerdts_2012a}.
\begin{assumption} \label{as:Sensitivity:OCP:smoothness}
      Let the functions $\varphi$, $l$, $\bff$, $\bg$ in \eqref{eq:intro:OCP} satisfy the following:
	\begin{itemize}
		\item[(i)] $\varphi$ is continuously differentiable.
		
		\item[(ii)] The mappings
		     $t \mapsto l \big(t, y, \bg(t, y) \big)$ and
		     $ t \mapsto \bff\big(t, y, \bg(t, y) \big)$
		are measurable for all $y \in \real^{n_y}$.
		
		\item[(iii)] The mappings
		     $y \mapsto l\big(t, y, \bg(t, y)\big)$ and
		     $ y \mapsto \bff\big(t, y, \bg(t, y)\big)$
		are uniformly continuously differentiable for $t \in I$.
		
		\item [(iv)] The first-order (total) derivatives
		     $(t, y) \mapsto \frac{d}{dy} l\big(t, y, \bg(t, y)\big)$ and
		     $(t, y) \mapsto \frac{d}{dy}\bff\big(t, y, \bg(t, y) \big)$
		are bounded in $I \times \real^{n_y}$.
	\end{itemize}
\end{assumption}

The first-order necessary optimality conditions for \eqref{eq:intro:OCP} are given in the following theorem.
\begin{theorem} \label{thm:Sensitivity:OCP:NOC}
  Let \cref{as:Sensitivity:OCP:smoothness} hold with $\bg = \overline{\bg}$.
  If  $\overline{\bx} \in  \big(W^{1, \infty}(I)\big)^{n_x}$, $\overline{\bu} \in  \big(L^\infty(I)\big)^{n_u}$, $\overline{\bp} \in \real^{n_p}$
  is a local minimum of \eqref{eq:intro:OCP} with $\bg = \overline{\bg}$, then
  there exists $\overline{\blambda} \in \big(W^{1,\infty}(I)\big)^{n_x}$ such that 
    \begin{align*}
    		\overline{\bx}'(t) &= \overline{\bff}[t], & t \in (t_0, t_f), \\ 
    	        \overline{\bx}(t_0) &= x_0, \\ 
    		\overline{\blambda}'(t) &= - \big(\overline{\bff}_x[t] + \overline{\bff}_g[t] \overline{\bg}_x[t] \big)^T \overline{\blambda}(t) - (\nabla_x \overline{l}[t] + \overline{\bg}_x[t]^T \nabla_g \overline{l}[t]), & t \in (t_0, t_f), \\  
    		\overline{\blambda}(t_f) &= \nabla_x \overline{\varphi}[t_f],\\ 
    		0 &=\big(\overline{\bff}_u[t] + \overline{\bff}_g[t] \overline{\bg}_u[t] \big)^T \overline{\blambda}(t) + (\nabla_u \overline{l}[t] + \overline{\bg}_u[t]^T \nabla_g \overline{l}[t]), & t \in (t_0, t_f), \\  
    		0 &= \nabla_p \overline{\varphi}[t_f] + \int_{t_0}^{t_f} \big(\overline{\bff}_p[t] + \overline{\bff}_g[t] \overline{\bg}_p[t] \big)^T \overline{\blambda}(t) \nonumber \\ &\quad + (\nabla_p \overline{l}[t] + \overline{\bg}_p[t]^T \nabla_g \overline{l}[t]) \, dt,
    \end{align*}
    using the shorthand
    \begin{equation} \label{eq:Sensitivity:OCP:shorthand-1}
    \begin{aligned}
         \overline{\bff}[\cdot] &:= \bff\Big( \cdot, \overline{\bx}(\cdot), \overline{\bu}(\cdot), \overline{\bp}, \, \overline{\bg} \big( \cdot, \overline{\bx}(\cdot), \overline{\bu}(\cdot), \overline{\bp} \big) \Big), & \qquad
          \overline{\bg}[\cdot] &:= \overline{\bg} \big( \cdot, \overline{\bx}(\cdot), \overline{\bu}(\cdot), \overline{\bp} \big), \\
    	\overline{l}[\cdot] &:=  l\Big( \cdot, \overline{\bx}(\cdot), \overline{\bu}(\cdot), \overline{\bp}, \, \overline{\bg} \big( \cdot, \overline{\bx}(\cdot), \overline{\bu}(\cdot), \overline{\bp} \big) \Big), & \qquad
    	\overline{\varphi}[t_f] &:= \varphi( \overline{\bx}(t_f), \bp )
    \end{aligned}
    \end{equation}
    and corresponding shorthand for partial derivatives.
\end{theorem}
\begin{proof}
	This result follows from \cite[Thm.~3.1.11]{MGerdts_2012a}.
	 Note that the surjectivity constraint qualification as described in \cite[Cor.~2.3.34]{MGerdts_2012a} is always satisfied for \eqref{eq:intro:OCP}.
\end{proof}

The optimality conditions of \cref{thm:Sensitivity:OCP:NOC} are considered as an operator equation to which the
Implicit Function \cref{thm:Sensitivity:intro:IFT} is applied to obtain sensitivities.
We continue to use the shorthand notation \eqref{eq:Sensitivity:OCP:shorthand-1}.
If we define the Banach spaces
\begin{equation} \label{eq:Sensitivity:OCP:function-spaces}
\begin{aligned}
	\cZ^\infty &:= \WW^{n_x} \times \LL^{n_u} \times \real^{n_p} \times \WW^{n_x}, \\
	\cV^\infty &:= \LL^{n_x} \times \real^{n_x} \times \LL^{n_x} \times \real^{n_x} \times \LL^{n_u} \times \real^{n_p}
\end{aligned}
\end{equation}
and the operator
\begin{subequations}     \label{eq:Sensitivity:OCP:KKT-operator}
\begin{equation}
		\bpsi : \cZ^\infty \times \big(\cG^3(I)\big)^{n_g} \rightarrow \cV^\infty
\end{equation}  
with
\begin{align}
    	& \bpsi(\overline{\bx}, \overline{\bu}, \overline{\bp}, \overline{\blambda}; \overline{\bg})  \nonumber \\
	&:= \begin{pmatrix} 
    		\overline{\bff}[\cdot] - \overline{\bx}'(\cdot) \\[0.5ex] 
    		x_0 - \overline{\bx}(t_0)  \\[0.5ex] 
    		\overline{\blambda}'(\cdot) + \big(\overline{\bff}_x[\cdot] 
		+ \overline{\bff}_g[\cdot] \overline{\bg}_x[\cdot] \big)^T \overline{\blambda}(\cdot) 
		+ \big(\nabla_x \overline{l}[\cdot] + \overline{\bg}_x[\cdot]^T \nabla_g \overline{l}[\cdot]\big)  \\[0.5ex] 
    		\nabla_x \overline{\varphi}[t_f] - \overline{\blambda}(t_f)  \\[0.5ex] 
    		\big(\overline{\bff}_u[\cdot] + \overline{\bff}_g[\cdot] \overline{\bg}_u[\cdot] \big)^T \overline{\blambda}(\cdot) 
		+ \big(\nabla_u \overline{l}[\cdot] + \overline{\bg}_u[\cdot]^T \nabla_g \overline{l}[\cdot]\big)  \\[0.5ex] 
    		 \nabla_p \overline{\varphi}[t_f] + \int_{t_0}^{t_f} \Big(\big(\overline{\bff}_p[t] 
		  + \overline{\bff}_g[t] \overline{\bg}_p[t] \big)^T \overline{\blambda}(t) + \big(\nabla_p \overline{l}[t] 
		  + \overline{\bg}_p[t]^T \nabla_g \overline{l}[t]\big) \Big) \, dt
		  \end{pmatrix},
    \end{align}
 \end{subequations}
the operator equation associated with the first-order necessary optimality conditions of \cref{thm:Sensitivity:OCP:NOC} 
is given by
\begin{equation}     \label{eq:Sensitivity:OCP:KKT-operator-eq}
    	\bpsi(\overline{\bx}, \overline{\bu}, \overline{\bp}, \overline{\blambda}; \overline{\bg})  = 0.
\end{equation}

To apply the  Implicit Function \cref{thm:Sensitivity:intro:IFT}, we first need to establish 
 the continuous Fr\'echet differentiability of $\bpsi$ in \eqref{eq:Sensitivity:OCP:KKT-operator}.
 This can be done by applying \cref{thm:Sensitivity:nemytskii:nemytskii-derivative}.
 However, because \eqref{eq:Sensitivity:OCP:KKT-operator} also involves the deriatives of
$\varphi$, $l$, and $\bff$ we need additional assumptions on $\varphi$, $l$, $\bff$.
Moreover, because \eqref{eq:Sensitivity:OCP:KKT-operator} also involves the deriatives of $\bg$,
the space $\big(\cG^3(I)\big)^{n_g}$ is needed for $\bg$.

\begin{assumption} \label{as:Sensitivity:OCP:smoothness-strong}
	 In addition to the properties listed in \cref{as:Sensitivity:OCP:smoothness}, let the following properties hold
	  for the function $\varphi$, $l$, and $\bff$  in \eqref{eq:intro:OCP}:
	\begin{itemize}
		\item[(i)] $\varphi$ is twice continuously differentiable.
		
		\item[(ii)] There exists $K$ such that for all $s \in \{ 0, 1, 2 \}$,
		\[
			\left\| \frac{\partial^s}{\partial (x, u, p, g)^s} \bff(t, 0, 0, 0, 0) \right\| \leq K, \qquad \aall t \in I.
		\]

		\item[(iii)] For all $R > 0$ there exists $L(R)$ such that for all $s \in \{ 0, 1, 2 \}$ it holds
		\begin{align*}
			&\left\| \frac{\partial^s}{\partial (x, u, p, g)^s} \bff(t, x_1, u_1, p_1, g_1) 
			         - \frac{\partial^s}{\partial (x, u, p, g)^s} \bff(t, x_2, u_2, p_2, g_2) \right\| \\ 
		        &\quad\leq L(R) \| (x_1, u_1, p_1, g_1) - (x_2, u_2, p_2, g_2) \|, \\ 
		        & \qquad \qquad \qquad \qquad \qquad \aall t \in I \mbox{ and all } (x_1, u_1, p_1, g_1), (x_2, u_2, p_2, g_2) \in \cB_R(0).
		\end{align*}
		
		\item[(iv)] Properties (ii) and (iii) also hold for $l$ (\emph{mutatis mutandis}).
	\end{itemize}
\end{assumption}
Properties (ii)-(iv) allow application of \cref{thm:Sensitivity:nemytskii:nemytskii-derivative} to establish continuous Fr\'echet differentiability of the Nemytskii operator \eqref{eq:Sensitivity:nemytskii:nemytskii-operator} for $\bphi = \bff$, $l$, and their first partial derivatives, which is required for condition (i) of the Implicit Function \cref{thm:Sensitivity:intro:IFT}, proven next.

\begin{lemma} \label{lem:Sensitivity:OCP:cts-frechet}
        Let $(\overline{\bx}, \overline{\bu}, \overline{\bp}, \overline{\blambda}) \in \cZ^\infty$ be given, let $\overline{\bg} \in \big(\cG^3(I)\big)^{n_g}$, 
	and define the Hamiltonian
	\[
		\overline{\BH}[\cdot] := \overline{l}[\cdot] + \overline{\blambda}(\cdot)^T \overline{\bff}[\cdot].
	\]

	If $\bphi = \bff, l$ satisfy the assumptions of \cref{thm:Sensitivity:nemytskii:nemytskii-derivative} and 
	if \cref{as:Sensitivity:OCP:smoothness-strong} holds,
	then $\bpsi$ in \eqref{eq:Sensitivity:OCP:KKT-operator} is continuously Fr\'echet differentiable in all arguments at the point 
	$(\overline{\bx}, \overline{\bu}, \overline{\bp}, \overline{\blambda}; \overline{\bg})$ with
	 partial Fr\'echet derivatives
\begin{subequations} \label{eq:Sensitivity:OCP:partial-frechet}
\begin{equation} \label{eq:Sensitivity:OCP:partial-frechet-a}
\begin{aligned}
	&\bpsi_{(\bx, \bu, \bp, {\tiny \blambda})}(\overline{\bx}, \overline{\bu}, \overline{\bp}, \overline{\blambda}; \overline{\bg})(\delta \bx, \delta \bu, \delta \bp, \delta \blambda) \\
	&\quad= 
	\begin{pmatrix} 
		\overline{\BA}[\cdot] \delta \bx(\cdot) + \overline{\BB}[\cdot] \delta \bu(\cdot) + \overline{\BC}[\cdot] \delta \bp - \delta \bx'(\cdot)  \\[0.75ex]
		- \delta \bx(t_0) \\[0.75ex]
		\delta \blambda'(\cdot) + \overline{\BH}_{xx}[\cdot] \delta \bx(\cdot) + \overline{\BH}_{xu}[\cdot] \delta \bu(\cdot) 
		  + \overline{\BH}_{xp}[\cdot] \delta \bp + \overline{\BA}[\cdot]^T \delta \blambda(\cdot)  \\[0.75ex]
		 \nabla_{xx}^2 \overline{\varphi}[t_f] \delta \bx(t_f) + \nabla_{xp}^2 \overline{\varphi}[t_f] \delta \bp - \delta \blambda(t_f)  \\[0.75ex]
		\overline{\BH}_{ux}[\cdot] \delta \bx(\cdot) + \overline{\BH}_{uu}[\cdot] \delta \bu(\cdot) 
		+ \overline{\BH}_{up}[\cdot] \delta \bp + \overline{\BB}[\cdot]^T \delta \blambda(\cdot)  \\[1.0ex]
		 \nabla_{px}^2 \overline{\varphi}[t_f]\delta \bx(t_f) +  \nabla_{pp}^2 \overline{\varphi}[t_f] \delta \bp \\ + \int_{t_0}^{t_f} \overline{\BH}_{px}[t] \delta \bx(t) + \overline{\BH}_{pu}[t] \delta \bu(t) + \overline{\BH}_{pp}[t] \delta \bp + \overline{\BC}[t]^T \delta \blambda(t) \, dt
	\end{pmatrix}
\end{aligned}
\end{equation}
and
\begin{equation} \label{eq:Sensitivity:OCP:partial-frechet-b}
\begin{aligned}
	&\bpsi_{\bg}(\overline{\bx}, \overline{\bu}, \overline{\bp}, \overline{\blambda}; \overline{\bg}) \delta \bg = 
	\begin{pmatrix} 
		 \overline{\bff}_g[\cdot] \delta \bg[\cdot]  \\[0.5ex]
		0  \\[0.5ex]
		 \overline{\BH}_{xg}[\cdot] \delta \bg[\cdot] + \delta \bg_x[\cdot]^T \overline{\bd}[\cdot]    \\[0.5ex]
		0  \\[0.5ex]
		 \overline{\BH}_{ug}[\cdot] \delta \bg[\cdot] + \delta \bg_u[\cdot]^T \overline{\bd}[\cdot]   \\[0.75ex]
		\int_{t_0}^{t_f}   \overline{\BH}_{pg}[t] \delta \bg[t] + \delta \bg_p[t]^T \overline{\bd}[t] \, dt
	\end{pmatrix},
\end{aligned}
\end{equation}
\end{subequations}
where
\begin{subequations} \label{eq:Sensitivity:OCP:LQOCP_components}
\begin{equation} \label{eq:Sensitivity:OCP:LQOCP_components1}
\begin{aligned}
	\overline{\BH}_{xx}[\cdot] 
	&= \nabla_{xx}^2 \overline{\BH}[\cdot] + \nabla_{xg}^2 \overline{\BH}[\cdot] \overline{\bg}_x[\cdot] + (\nabla_{xg}^2 \overline{\BH}[\cdot] \overline{\bg}_x[\cdot])^T + \overline{\bg}_x[\cdot]^T \nabla_{gg}^2 \overline{\BH}[\cdot] \overline{\bg}_x[\cdot], \\
	\overline{\BH}_{xu}[\cdot] 
	&= \nabla_{xu}^2 \overline{\BH}[\cdot] + \nabla_{xg}^2 \overline{\BH}[\cdot] \overline{\bg}_u[\cdot] + (\nabla_{xg}^2 \overline{\BH}[\cdot] \overline{\bg}_u[\cdot])^T + \overline{\bg}_x[\cdot]^T \nabla_{gg}^2 \overline{\BH}[\cdot] \overline{\bg}_u[\cdot], \\
	\overline{\BH}_{xp}[\cdot]
	&= \nabla_{xp}^2 \overline{\BH}[\cdot] + \nabla_{xg}^2 \overline{\BH}[\cdot] \overline{\bg}_p[\cdot] + (\nabla_{xg}^2 \overline{\BH}[\cdot] \overline{\bg}_p[\cdot])^T + \overline{\bg}_x[\cdot]^T \nabla_{gg}^2 \overline{\BH}[\cdot] \overline{\bg}_p[\cdot], \\
	\overline{\BH}_{uu}[\cdot] 
	&= \nabla_{uu}^2 \overline{\BH}[\cdot] + \nabla_{ug}^2 \overline{\BH}[\cdot] \overline{\bg}_u[\cdot] + (\nabla_{ug}^2 \overline{\BH}[\cdot] \overline{\bg}_u[\cdot])^T + \overline{\bg}_u[\cdot]^T \nabla_{gg}^2 \overline{\BH}[\cdot] \overline{\bg}_u[\cdot], \\
	\overline{\BH}_{up}[\cdot] 
	&= \nabla_{up}^2 \overline{\BH}[\cdot] + \nabla_{ug}^2 \overline{\BH}[\cdot] \overline{\bg}_p[\cdot] + (\nabla_{ug}^2 \overline{\BH}[\cdot] \overline{\bg}_p[\cdot])^T + \overline{\bg}_u[\cdot]^T \nabla_{gg}^2 \overline{\BH}[\cdot] \overline{\bg}_p[\cdot], \\
	\overline{\BH}_{pp}[\cdot]
	&= \nabla_{pp}^2 \overline{\BH}[\cdot] + \nabla_{pg}^2 \overline{\BH}[\cdot] \overline{\bg}_p[\cdot] + (\nabla_{pg}^2 \overline{\BH}[\cdot] \overline{\bg}_p[\cdot])^T + \overline{\bg}_p[\cdot]^T \nabla_{gg}^2 \overline{\BH}[\cdot] \overline{\bg}_p[\cdot], \\
	\overline{\BH}_{xg}[\cdot]
	&= \nabla_{xg}^2 \overline{\BH}[\cdot] + \overline{\bg}_x[\cdot]^T \nabla_{gg}^2 \overline{\BH}[\cdot], \\
	\overline{\BH}_{ug}[\cdot]
	&= \nabla_{ug}^2 \overline{\BH}[\cdot] + \overline{\bg}_u[\cdot]^T \nabla_{gg}^2 \overline{\BH}[\cdot], \\
	\overline{\BH}_{pg}[\cdot]
	&= \nabla_{pg}^2 \overline{\BH}[\cdot] + \overline{\bg}_p[\cdot]^T \nabla_{gg}^2 \overline{\BH}[\cdot]
\end{aligned}
\end{equation}
and
\begin{equation} \label{eq:Sensitivity:OCP:LQOCP_components2}
\begin{aligned}
    \overline{\BA}[\cdot] &= \overline{\bff}_x[\cdot] + \overline{\bff}_g[\cdot] \overline{\bg}_x[\cdot], \quad
	\overline{\BB}[\cdot] = \overline{\bff}_u[\cdot] + \overline{\bff}_g[\cdot] \overline{\bg}_u[\cdot], \quad
	\overline{\BC}[\cdot] = \overline{\bff}_p[\cdot] + \overline{\bff}_g[\cdot] \overline{\bg}_p[\cdot], \\
	\overline{\bd}[\cdot] &= \overline{\bff}_g[\cdot]^T \overline{\blambda}(\cdot).
\end{aligned}
\end{equation}
\end{subequations}
\end{lemma}
\begin{proof}
       Provided $\bpsi$ is continuously Fr\'echet differentiable, the form of \eqref{eq:Sensitivity:OCP:partial-frechet-a} and 
	\eqref{eq:Sensitivity:OCP:partial-frechet-b} follows from \cref{thm:Sensitivity:nemytskii:nemytskii-derivative} 
	and repeated application of the product rule and the chain rule. It remains to establish the continuous Fr\'echet differentiability of $\bpsi$. 
	Because products and sums preserve continuous Fr\'echet differentiability, it suffices to show the continuous Fr\'echet differentiability 
	of each term in the enumeration of $\bpsi$. 
	Since many terms are structurally similar, we do not provide an exhaustive proof for each term, which would be tedious.
	Instead, we only provide a proof for those terms for which continuous Fr\'echet differentiability is most difficult to establish.

	\cref{thm:Sensitivity:nemytskii:nemytskii-derivative} can be applied to establish continuous Fr\'echet differentiability 
	of several terms in \eqref{eq:Sensitivity:OCP:KKT-operator}. 
	For instance, consider the term $\overline{\bff}[\cdot]$
         appearing in the first component of $\bpsi$. Recall from \eqref{eq:Sensitivity:OCP:shorthand-1} and 
         \eqref{eq:Sensitivity:OCP:function-spaces} that
         \begin{equation*}
      	        \overline{\bff}[\cdot] = \bff\Big( \cdot, \overline{\by}(\cdot), \overline{\bg} \big( \cdot, \overline{\by}(\cdot) \big) \Big),
        \end{equation*}
         where $\overline{\by} = (\overline{\bx}, \overline{\bu}, \overline{\bp}) \in \cY^\infty 
         = \big( W^{1, \infty}(I) \big)^{n_x} \times \big( L^\infty(I) \big)^{n_u} \times \real^{n_p}$.
         Clearly $\overline{\bff}[\cdot]$ is continuously Fr\'echet differentiable by \cref{thm:Sensitivity:nemytskii:nemytskii-derivative}. 
         
         Now, consider the term 
         \[
                 \overline{\bff}_x[\cdot] = \bff_x\Big( \cdot, \overline{\by}(\cdot), \overline{\bg}\big(\cdot, \overline{\by}(\cdot) \big) \Big) 
         \] 
         appearing in the third component of $\bpsi$. 
         By flattening the matrix-valued output of $\bff_x$, one obtains a vector-valued function to which 
         \cref{thm:Sensitivity:nemytskii:nemytskii-derivative} may be applied to establish continuous 
         Fr\'echet differentiability of the matrix-valued function.
          By similar arguments, $\overline{\bff}_g[\cdot]$ is also continuously Fr\'echet differentiable.
          
          It is easily shown that the hypotheses of \cref{thm:Sensitivity:nemytskii:nemytskii-derivative} hold for
            \[
              	\bphi(t, y, g) := g.
           \]
          Thus, the Nemytskii operator
       \begin{equation} \label{eq:Sensitivity:OCP:nemytskii-g-x}
    	\LL^{n_x + n_u + n_p} \times \big(\cG^2(I)\big)^{n_g n_x} \ni (\by, \bg_x) 
	        \mapsto \BPhi(\by, \bg_x) := \bg_x\big(\cdot, \by(\cdot)\big) \in \LL^{n_g n_x}
        \end{equation}
        is continuously Fr\'echet differentiable.
        Since the map $\big(\cG^3(I)\big)^{n_g} \ni \bg \mapsto \bg_x \in \big(\cG^2(I)\big)^{n_g n_x}$
        is clearly bounded and linear, and since $\cY^\infty$ is continuously embedded  in 
        $\LL^{n_x + n_u + n_p}$, it follows from the differentiability of \eqref{eq:Sensitivity:OCP:nemytskii-g-x} 
        that $\overline{\bg}_x[\cdot]$ is continuously Fr\'echet differentiable.

        Similar arguments may be used to establish continuous Fr\'echet differentiability of the other terms in $\bpsi$. 
\end{proof}

\begin{remark} \label{rk:Sensitivity:OCP:regularity-of-terms}
	Note that all component functions in \eqref{eq:Sensitivity:OCP:LQOCP_components} are essentially bounded due to 
	conditions (ii)-(iv) of \cref{as:Sensitivity:OCP:smoothness-strong} and the definition of $\big(\cG^3(I)\big)^{n_g}$. 
\end{remark}

The continuous invertibility of 
$\bpsi_{(\bx, \bu, \bp, {\tiny \blambda})}(\overline{\bx}, \overline{\bu}, \overline{\bp}, \overline{\blambda}; \overline{\bg})$ required by
hypothesis (ii) of the Implicit Function \cref{thm:Sensitivity:intro:IFT}
is proven using a second-order sufficient optimality condition.
Second-order sufficient conditions are more naturally formulated in $L^2$ spaces where Hilbert space structures can be leveraged, 
but proving Fr\'echet differentiability requires the stronger topology of $L^\infty$ spaces; this is related to the so-called \emph{two-norm discrepancy} discussed in e.g., \cite{KMalanowski_1993a}, \cite{HMaurer_1981}, \cite{ECasas_FTroltzsch_2015a}. The following second-order sufficient optimality condition is assumed to deal with this discrepancy.
\begin{assumption} \label{as:Sensitivity:OCP:SSOC}
	\begin{itemize}
	\item[(i)] The matrix
	\[
		              \begin{pmatrix}  \nabla_{xx}^2 \overline{\varphi}[t_f] &  \nabla_{xp}^2 \overline{\varphi}[t_f] \\ 
		                                      \nabla_{px}^2 \overline{\varphi}[t_f] &  \nabla_{pp}^2 \overline{\varphi}[t_f] 
		           \end{pmatrix} 
	\]
	is symmetric positive semidefinite.		
	\item[(ii)] There exists $\epsilon > 0$ such that
	\begin{align*}
		&\int_{t_0}^{t_f} \begin{pmatrix} \delta \bx(t) \\ \delta \bu(t) \\ \delta \bp \end{pmatrix}^T 
		\begin{pmatrix} \overline{\BH}_{xx}[t] & \overline{\BH}_{xu}[t] & \overline{\BH}_{xp}[t] \\ \overline{\BH}_{ux}[t] & \overline{\BH}_{uu}[t] & \overline{\BH}_{up}[t] \\ \overline{\BH}_{px}[t] & \overline{\BH}_{pu}[t] & \overline{\BH}_{pp}[t] \end{pmatrix}
		\begin{pmatrix} \delta \bx(t) \\ \delta \bu(t) \\ \delta \bp \end{pmatrix} \, dt \\ 
		&\geq \epsilon \, \| (\delta \bx, \delta \bu, \delta \bp) \|_{( L^2(I) )^{n_x} \times ( L^2(I) )^{n_u} \times \real^{n_p}}^2
	\end{align*}
	for all $(\delta \bx, \delta \bu, \delta \bp) \in \big( W^{1, 2}(I) \big)^{n_x} \times \big( L^2(I) \big)^{n_u} \times \real^{n_p}$ satisfying
	\begin{align*}
		\delta \bx'(t) &= \overline{\BA}[t] \delta \bx(t) + \overline{\BB}[t] \delta \bu(t) + \overline{\BC}[t] \delta \bp, & \aall t \in I, \\
	\delta \bx(t_0) &= 0.
	\end{align*}	
	
	\item[(iii)] The inverse $\overline{\BH}_{uu}[t]^{-1}$ exists for almost all $t \in I$, and there exists $M > 0$ such that
	\[
		\big\| \overline{\BH}_{uu}[t]^{-1} \big\| \leq M, \qquad \aall t \in I.
	\]
	\end{itemize}
\end{assumption}

Condition (iii), which is adapted from \cite[$\widetilde{\rm SSC}$, p.~224]{KMalanowski_1997a}, ensures existence and uniqueness of solutions to a particular linear quadratic optimal control problem in the Hilbert space 
\[
	\cZ^2 := \big( W^{1, 2}(I) \big)^{n_x} \times \big( L^2(I) \big)^{n_u} \times \real^{n_p} \times \big( W^{1, 2}(I) \big)^{n_x}
\]
whose solution $(\delta \bx, \delta \bu, \delta \bp, \delta \blambda)$ gives the sensitivity of the solution to \eqref{eq:intro:OCP}, 
and condition (iii), which follows from a Legendre-Clebsch condition like e.g., \cite[Eq.~4.9, p.~217]{KMalanowski_1997a}, 
ensures that this unique solution also belongs to $\cZ^\infty$, which is the original function space setting.

The next result establishes hypothesis (ii) for the Implicit Function \cref{thm:Sensitivity:intro:IFT}.
\begin{lemma} \label{lem:Sensitivity:OCP:bijection}
	If the assumptions of \cref{lem:Sensitivity:OCP:cts-frechet} and \cref{as:Sensitivity:OCP:SSOC} hold, 
	then the partial Fr\'echet derivative 
	$\bpsi_{(\bx, \bu, \bp, {\tiny \blambda})}(\overline{\bx}, \overline{\bu}, \overline{\bp}, \overline{\blambda}; \overline{\bg})$ 
	in \eqref{eq:Sensitivity:OCP:partial-frechet-a} is a bijection from $\cZ^\infty \times \big(\cG^3(I)\big)^{n_g}$ to $\cV^\infty$.
\end{lemma}
\begin{proof}
    Given a right-hand side $\bv := (\br, r_0, \bc_x, \sigma_f, \bc_u, \gamma) \in \cV^\infty$ the equation 
    \begin{equation}     \label{eq:Sensitivity:OCP:LQOCP0}
           \bpsi_{(\bx, \bu, \bp, {\tiny \blambda})}(\overline{\bx}, \overline{\bu}, \overline{\bp}, \overline{\blambda}; \overline{\bg})
             (\delta \bx, \delta \bu, \delta \bp, \delta \blambda) = -\bv
    \end{equation}
     is equivalent to the first-order necessary optimality conditions of the linear quadratic optimal control problem
    \begin{equation}     \label{eq:Sensitivity:OCP:LQOCP}
    \begin{aligned}
    	\min_{\delta \bx, \delta \bu, \delta \bp} \quad & \int_{t_0}^{t_f} \begin{pmatrix} \bc_x(t) \\ \bc_u(t) \end{pmatrix}^T \begin{pmatrix} \delta \bx(t) \\ \delta \bu(t) \end{pmatrix} \, dt  \\
    	& + \frac{1}{2} \int_{t_0}^{t_f} \begin{pmatrix} \delta \bx(t) \\ \delta \bu(t) \\ \delta \bp \end{pmatrix}^T \begin{pmatrix} \overline{\BH}_{xx}[t] & \overline{\BH}_{xu}[t] & \overline{\BH}_{xp}[t] \\ \overline{\BH}_{ux}[t] & \overline{\BH}_{uu}[t] & \overline{\BH}_{up}[t] \\ \overline{\BH}_{px}[t] & \overline{\BH}_{pu}[t] & \overline{\BH}_{pp}[t] \end{pmatrix} \begin{pmatrix} \delta \bx(t) \\ \delta \bu(t) \\ \delta \bp \end{pmatrix} \, dt \\
    		& + \begin{pmatrix} \sigma_f \\ \gamma \end{pmatrix}^T \begin{pmatrix} \delta \bx(t_f) \\ \delta \bp \end{pmatrix} 
    		  + \frac{1}{2} \begin{pmatrix} \delta \bx(t_f) \\ \delta \bp \end{pmatrix}^T
    		            \begin{pmatrix}  \nabla_{xx}^2 \overline{\varphi}[t_f] &  \nabla_{xp}^2 \overline{\varphi}[t_f] \\ 
    		                                      \nabla_{px}^2 \overline{\varphi}[t_f] &  \nabla_{pp}^2 \overline{\varphi}[t_f] 
    		           \end{pmatrix} 
    		           \begin{pmatrix} \delta \bx(t_f) \\ \delta \bp \end{pmatrix} \\[1ex]
            \mbox{s.t.} \quad
    	& \delta \bx'(t) = \overline{\BA}[t] \delta \bx(t) + \overline{\BB}[t] \delta \bu(t) + \overline{\BC}[t] \delta \bp + \br(t), \hspace{3cm} \aall t \in I, \\
    	& \delta \bx(t_0) = r_0
    \end{aligned}
    \end{equation}
    with costate $\delta \blambda$.
    Accordingly, to ensure bijectivity of $\bpsi_{(\bx, \bu, \bp, {\tiny \blambda})}(\overline{\bx}, \overline{\bu}, \overline{\bp}, \overline{\blambda}; \overline{\bg})$, it suffices to establish existence and uniqueness of solutions to \eqref{eq:Sensitivity:OCP:LQOCP} in $\cZ^\infty$. The proof proceeds as follows: existence and uniqueness of solutions will be established in $\cZ^2$ using \cref{as:Sensitivity:OCP:SSOC} (i)-(ii), then \cref{as:Sensitivity:OCP:SSOC} (iii) will be used to show that the solution also belongs to $\cZ^\infty$.
    
    Taking \eqref{eq:Sensitivity:OCP:LQOCP} as a problem in 
    \[
    	\cY^2 := \big( W^{1, 2}(I) \big)^{n_x} \times \big( L^2(I) \big)^{n_u} \times \real^{n_p},
    \]
    the objective function is continuous and coercive over the feasible set by \cref{as:Sensitivity:OCP:SSOC} (i)-(ii) and the affine linearity of the constraints, and the feasible set is closed and convex; thus, by \cite[Cor.~3.23]{HBrezis_2011a} there exists a solution $(\delta \bx, \delta \bu, \delta \bp) \in \cY^2$, and furthermore the strong convexity of the objective function over the feasible set implies that the solution is unique. The costate is then given by the unique solution $\delta \blambda \in \big( W^{1, 2}(I) \big)^{n_x}$ of the adjoint equation
    \begin{equation} \label{eq:Sensitivity:OCP:adjoint}
    \begin{aligned}
    	-\delta \blambda'(t) &= \overline{\BH}_{xx}[t] \delta \bx(t) + \overline{\BH}_{xu}[t] \delta \bu(t) + \overline{\BH}_{xp}[t] \delta \bp + \overline{\BA}[t]^T \delta \blambda(t) + \bc_x(t), & \aall t \in I, \\
    	\delta \blambda(t_f) &= \overline{\BH}_{xx}^f \delta \bx(t_f) + \gamma.
    \end{aligned}
    \end{equation}
    Note that the terms $\overline{\BH}_{xx}[\cdot]$, $\overline{\BA}[\cdot]$, etc. have the necessary regularity for this equation to be well-posed in the $L^2$ setting (Remark~\ref{rk:Sensitivity:OCP:regularity-of-terms}).

    It remains to show that the full solution $\delta \bz = (\delta \bx, \delta \bu, \delta \bp, \delta \blambda)$ of \eqref{eq:Sensitivity:OCP:LQOCP} in $\cZ^2$ also belongs to $\cZ^\infty$, thereby establishing that \eqref{eq:Sensitivity:OCP:LQOCP} has a unique solution in $\cZ^\infty$. 
    Using \eref{eq:Sensitivity:OCP:partial-frechet-a}, the fifth equation in \eqref{eq:Sensitivity:OCP:LQOCP0} and
    \cref{as:Sensitivity:OCP:SSOC} (iii) imply
    \[
    	\delta \bu(t) = \overline{\BH}_{uu}[t]^{-1} \Big( \bc_u(t) - \overline{\BH}_{ux}[t] \delta \bx(t) - \overline{\BH}_{up}[t] \delta \bp - \overline{\BB}[t]^T \delta \blambda(t) \Big), \qquad \aall t \in I.
    \]
    Since all functions on the right-hand side of the previous equation are essentially bounded, $\delta \bu \in \big(L^\infty(I) \big)^{n_u}$. 
    From this, it is readily observed that the state equation given by
    \begin{align*}
    \delta \bx'(t) &= \overline{\BA}[t] \delta \bx(t) + \overline{\BB}[t] \delta \bu(t) + \overline{\BC}[t] \delta \bp + \br(t), & \aall t \in I, \\
    	\delta \bx(t_0) &= r_0
    \end{align*}
    has its solution $\delta \bx$ in $\big( W^{1,\infty}(I) \big)^{n_x}$ due to the essential boundedness of the control 
    $\delta \bu$, the matrix-valued functions $\overline{\BA}[\cdot]$, $\overline{\BB}[\cdot]$, $\overline{\BC}[\cdot]$, and 
    the vector-valued function $\br \in \big(L^\infty(I) \big)^{n_x}$. Observe now that $\delta \blambda \in \big( W^{1,\infty}(I) \big)^{n_x}$ 
    follows from the costate equation \eqref{eq:Sensitivity:OCP:adjoint} using the newly derived smoothness properties of 
    $\delta \bx$ and $\delta \bu$. Therefore, the complete solution $\delta \bz = (\delta \bx, \delta \bu, \delta \bp, \delta \blambda) \in \cZ^\infty$ 
    uniquely solves \eqref{eq:Sensitivity:OCP:LQOCP}, thereby establishing bijectivity of 
    $\bpsi_{(\bx, \bu, \bp, {\tiny \blambda})}(\overline{\bx}, \overline{\bu}, \overline{\bp}, \overline{\blambda}; \overline{\bg})$ 
    as a mapping from $\cZ^\infty \times \big(\cG^3(I)\big)^{n_g}$ to $\cV^\infty$.
\end{proof}

\cref{lem:Sensitivity:OCP:cts-frechet} and~\cref{lem:Sensitivity:OCP:bijection} allow application of the Implicit Function \cref{thm:Sensitivity:intro:IFT} to obtain a sensitivity equation that gives the perturbation $\delta \bz = (\delta \bx, \delta \bu, \delta \bp, \delta \blambda) \in \cZ^\infty$ of the solution (including the costate) to \eqref{eq:intro:OCP} when perturbing the model by $\delta \bg \in \big(\cG^3(I)\big)^{n_g}$.
\begin{theorem}   \label{thm:Sensitivity:OCP:OCP-solution-sensitivity} 
       Let $\overline{\bz} := (\overline{\bx}, \overline{\bu}, \overline{\bp}, \overline{\blambda}) \in \cZ^\infty$ and 
        $\overline{\bg} \in \big(\cG^3(I)\big)^{n_g}$ satisfy $\bpsi(\overline{\bz}; \overline{\bg}) = 0$ with $\bpsi$ 
        as in \eqref{eq:Sensitivity:OCP:KKT-operator}. 
	If the assumptions of \cref{lem:Sensitivity:OCP:cts-frechet} and~\cref{lem:Sensitivity:OCP:bijection} hold, 
	then there exist neighborhoods
	\begin{align*}
		&\cN(\overline{\bg}) \subset \big(\cG^3(I)\big)^{n_g}, \\
		&\cN(\overline{\bx}) \subset \big( W^{1,\infty}(I) \big)^{n_x}, \quad
		\cN(\overline{\bu}) \subset \big( L^\infty(I) \big)^{n_u}, \quad
		\cN(\overline{\bp}) \subset \real^{n_p}, \quad
		\cN(\overline{\blambda}) \subset \big( W^{1,\infty}(I) \big)^{n_x}
	\end{align*}
	and unique continuously Fr\'echet differentiable mappings
	\begin{align*}
		\bx : \cN(\overline{\bg}) \rightarrow \cN(\overline{\bx}), \quad
		\bu : \cN(\overline{\bg}) \rightarrow \cN(\overline{\bu}), \quad
		\bp : \cN(\overline{\bg}) \rightarrow \cN(\overline{\bp}), \quad
		\blambda : \cN(\overline{\bg}) \rightarrow \cN(\overline{\blambda})
	\end{align*}
	satisfying
	\begin{align*}
		&\bx(\overline{\bg}) = \overline{\bx}, \quad
		\bu(\overline{\bg}) = \overline{\bu}, \quad
		\bp(\overline{\bg}) = \overline{\bp}, \quad
		\blambda(\overline{\bg}) = \overline{\blambda}, \\
		&\bpsi\big( \bx(\bg), \bu(\bg), \bp(\bg), \blambda(\bg); \bg \big) = 0 \qquad \qquad \mbox{ for all } \; \bg \in \cN(\overline{\bg}).
	\end{align*}
	
	Moreover, the mapping $\bz(\bg) := \big(\bx(\bg), \bu(\bg), \bp(\bg), \blambda(\bg) \big)$ is continuously Fr\'echet differentiable at $\overline{\bg}$ 
         and the Fr\'echet derivative
         \[
                   \delta \bz = \bz_\bg(\overline{\bg}) \delta \bg = (\delta \bx, \delta \bu, \delta \bp, \delta \blambda) \in \cZ^\infty
         \]
          is the unique solution of 
	 \begin{subequations}   \label{eq:Sensitivity:OCP:ocp-solution-sensitivity-system} 
	 \begin{align}
	        \delta \bx'(t) &= \overline{\BA}[t] \delta \bx(t) + \overline{\BB}[t] \delta \bu(t) + \overline{\BC}[t] \delta \bp  
	                                        + \overline{\bff}_g[t] \delta \bg[t], & \aall t \in (t_0, t_f), \\ 
		\delta \bx(t_0) &= 0, \\ 
		\delta \blambda'(t) &= -  \overline{\BA}[t]^T \delta \blambda(t) -\overline{\BH}_{xx}[t] \delta \bx(t) - \overline{\BH}_{xu}[t] \delta \bu(t) \nonumber \\ 
		&\quad - \overline{\BH}_{xp}[t] \delta \bp 
		                                         -  \overline{\BH}_{xg}[t] \delta \bg[t] - \delta \bg_x[t]^T \overline{\bd}[t], & \aall t \in (t_0, t_f), \\ 
		 \delta \blambda(t_f)  &= \nabla_{xx}^2 \overline{\varphi}[t_f] \delta \bx(t_f) + \nabla_{xp}^2 \overline{\varphi}[t_f] \delta \bp,
           \end{align}
           \begin{align}
		&\overline{\BH}_{ux}[t] \delta \bx(t) + \overline{\BH}_{uu}[t] \delta \bu(t) + \overline{\BH}_{up}[t] \delta \bp + \overline{\BB}[t]^T \delta \blambda(t) \nonumber \\
		&= - \overline{\BH}_{ug}[t] \delta \bg[t] - \delta \bg_u[t]^T \overline{\bd}[t],  & \aall t \in (t_0, t_f), 
           \end{align}
           \begin{align}
	         & \nabla_{px}^2 \overline{\varphi}[t_f]\delta \bx(t_f)  
	            + \int_{t_0}^{t_f} \big( \overline{\BH}_{px}[t] \delta \bx(t) + \overline{\BH}_{pu}[t] \delta \bu(t) + \BC[t]^T \delta \blambda(t)  \big) \, dt \nonumber \\
	            &\quad +  \Big( \nabla_{pp}^2 \overline{\varphi}[t_f] +  \int_{t_0}^{t_f}  \overline{\BH}_{pp}[t]  \, dt \Big)  \delta \bp
	        = - \int_{t_0}^{t_f}   \overline{\BH}_{pg}[t] \delta \bg[t] +  \delta \bg_p[t]^T \overline{\bd}[t] \, dt.
         \end{align}
         \end{subequations}
         
         The system \eqref{eq:Sensitivity:OCP:ocp-solution-sensitivity-system}  constitutes the necessary and sufficient optimality conditions for the linear quadratic optimal control
         problem \eqref{eq:Sensitivity:OCP:LQOCP} with
          \begin{subequations}   \label{eq:Sensitivity:OCP:ocp-solution-sensitivity-system-linear} 
	 \begin{align}
		\br(t) &= \overline{\bff}_g[t] \delta \bg[t],  &
		 \begin{pmatrix} \bc_x(t) \\ \bc_u(t) \end{pmatrix} 
		 &= \begin{pmatrix} \overline{\BH}_{xg}[t] \delta \bg[t] + \delta \bg_x[t]^T \overline{\bd}[t] \\  
		                             \overline{\BH}_{ug}[t] \delta \bg[t] + \delta \bg_u[t]^T \overline{\bd}[t]
		        \end{pmatrix}, \\
		r_0 &= \sigma_f = 0, &
		 \gamma &= \int_{t_0}^{t_f}   \overline{\BH}_{pg}[t] \delta \bg[t] + \delta \bg_p[t]^T \overline{\bd}[t] \, dt.
	 \end{align}
	 \end{subequations}
\end{theorem}

\subsection{Fr\'echet Differentiability of the QoI}      \label{sec:sensitivity-QoI}

Next, consider the sensitivity of the QoI \eqref{eq:intro:QoI} defined in the spaces
\begin{equation} \label{eq:Sensitivity:OCP:QoI-tilde}
\begin{aligned}
	q &:  \cY^\infty  \times \big(\cG^3(I)\big)^{n_g} \rightarrow \real, \\
	\widetilde{q} &: \big(\cG^3(I)\big)^{n_g} \rightarrow \real.
\end{aligned}
\end{equation}
The continuously Fr\'echet differentiability of \eqref{eq:Sensitivity:OCP:QoI-tilde} is a consequence of
the chain rule and is stated in the following \cref{thm:Sensitivity:OCP:sensitivity_result_qoi}.
More importantly, the application of the  Fr\'echet derivative of $\widetilde{q}$ applied to any $\delta \bg$
can be efficiently computed using the adjoint equation approach stated in
\cref{thm:Sensitivity:OCP:sensitivity_result_qoi-adjoint} below.

\begin{theorem} \label{thm:Sensitivity:OCP:sensitivity_result_qoi}
	If \cref{as:Sensitivity:OCP:smoothness-strong} holds with $\varphi$ and $l$ replaced by $\phi$ and $\ell$ respectively, 
	then $q$ in \eqref{eq:intro:QoI} is continuously Fr\'echet differentiable in the spaces \eqref{eq:Sensitivity:OCP:QoI-tilde}, and its derivative is given by
	\begin{align} \label{eq:Sensitivity:OCP:q-Frechet_deriv}
		 &  q'(\overline{\bx}, \overline{\bu}, \overline{\bp}; \overline{\bg})(\delta \bx, \delta \bu, \delta \bp, \delta \bg)    \nonumber \\
		 &= \nabla_x  \overline{\phi}[t_f]^T \delta \bx(t_f)  +  \nabla_p  \overline{\phi}[t_f]^T \delta  \bp   \nonumber \\
		 &\quad  + \int_{t_0}^{t_f} \big( \nabla_x \overline{\ell}[t] + \overline{\bg}_x[t]^T \nabla_g \overline{\ell}[t]  \big)^T  \delta \bx(t)  
		                                            + \big(  \nabla_u \overline{\ell}[t] + \overline{\bg}_u[t]^T \nabla_g \overline{\ell}[t] \big)^T \delta \bu(t)   \nonumber \\
		 &\quad\qquad   + \big( \nabla_p \overline{\ell}[t] + \overline{\bg}_p[t]^T \nabla_g \overline{\ell}[t] \big)^T \delta \bp 
		                          +  \nabla_g \overline{\ell}[t]^T \delta \bg[t] \, dt,
	\end{align}
	using the shorthand
       \begin{equation*}
        \overline{\ell}[\cdot] :=  \ell\Big( \cdot, \overline{\bx}(\cdot), \overline{\bu}(\cdot), \overline{\bp}, \, \overline{\bg} \big( \cdot, \overline{\bx}(\cdot), \overline{\bu}(\cdot), \overline{\bp} \big) \Big),  \qquad
    	\overline{\phi}[t_f] := \phi\big( \overline{\bx}(t_f), \overline{\bp} \big)
       \end{equation*}
       and corresponding shorthand for partial derivatives.
       
	If, in addition, the assumptions of \cref{thm:Sensitivity:OCP:OCP-solution-sensitivity} hold, then $\widetilde{q}$ in \eqref{eq:intro:QoI} is continuously Fr\'echet differentiable in the spaces \eqref{eq:Sensitivity:OCP:QoI-tilde}, and its derivative
	at $\overline{\bg}$  is given by
         \begin{align*}
	       \widetilde{q}_\bg(\overline{\bg}) \delta \bg 
	       =  q'\big(\overline{\bx}(\cdot \, ; \overline{\bg}), \overline{\bu}(\cdot \, ; \overline{\bg}), \overline{\bp}(\overline{\bg}), \overline{\bg} \big) 
	               (\delta \bx, \delta \bu, \delta \bp, \delta \bg),
	\end{align*}
	where $\big(\overline{\bx}(\cdot \, ; \overline{\bg})$, $\overline{\bu}(\cdot \, ; \overline{\bg})$, $\overline{\bp}(\overline{\bg})\big) \in \cY^\infty$ is the solution of
        the optimal control problem \eqref{eq:intro:OCP} with model $\bg = \overline{\bg} \in \big(\cG^3(I)\big)^{n_g}$ 
        and $(\delta \bx, \delta \bu, \delta \bp, \delta \blambda) \in \cZ^\infty$ is the solution of \eqref{eq:Sensitivity:OCP:ocp-solution-sensitivity-system}.
\end{theorem}
\begin{proof}
The differentiability of the first term in \eqref{eq:intro:QoI-bolza} is obvious. The continuous Fr\'echet differentiability of the second term follows from the continuous Fr\'echet differentiability of the Nemytskii operator 
\[
	\cY^\infty \times \big(\cG^3(I)\big)^{n_g} \ni (\by, \bg) \mapsto \ell\Big( \cdot, \by(\cdot), \bg \big( \cdot, \by(\cdot) \big) \Big) \in L^\infty(I)
\] 
as a consequence of \cref{thm:Sensitivity:nemytskii:nemytskii-derivative} and the fact that integration over $I$ defines a bounded linear functional on $L^\infty(I)$. This proves the first part of the theorem. The second part immediately follows from the chain rule, with the sensitivity of the OCP solution given by the solution of the system of equations \eqref{eq:Sensitivity:OCP:ocp-solution-sensitivity-system} in \cref{thm:Sensitivity:OCP:OCP-solution-sensitivity}.
\end{proof}

Finally, the following theorem uses adjoints to compute the Fr\'echet derivative of $\widetilde{q}$ in \eqref{eq:Sensitivity:OCP:QoI-tilde} without computing the sensitivity of the OCP solution.
\begin{theorem} \label{thm:Sensitivity:OCP:sensitivity_result_qoi-adjoint}
	If the assumptions of \cref{thm:Sensitivity:OCP:sensitivity_result_qoi} hold, 
    then 
    \begin{align} \label{eq:Sensitivity:OCP:sensitivity_result_qoi_adjoint}
    	\widetilde{q}_\bg(\overline{\bg}) \delta \bg 
	= &  \int_{t_0}^{t_f} \!\! \Big(  \overline{\BH}_{gx}[t]  \widetilde{\delta \bx}(t) 
                                                  + \overline{\BH}_{gu}[t]  \widetilde{\delta \bu}(t)
                                                  + \overline{\BH}_{gp}[t]  \widetilde{\delta \bp}
                                                  + \overline{\bff}_{g}[t]^T \widetilde{\delta \blambda}(t) 
                                                  +   \nabla_g \overline{\ell}[t]  \Big)^T  \delta \bg[t]   \, dt	    \nonumber \\ 
            &   +  \int_{t_0}^{t_f}  \Big(  \overline{\bd}[t]^T \delta \bg_x[t] \widetilde{\delta \bx}(t) \Big)  
                                             +  \Big(  \overline{\bd}[t]^T \delta \bg_u[t] \widetilde{\delta \bu}(t) \Big) 
                                             +  \Big(  \overline{\bd}[t]^T \delta \bg_p[t] \widetilde{\delta \bp} \Big)  \, dt,
    \end{align}
          where 
                   $(\widetilde{\delta \bx}, \; \widetilde{\delta \bu}, \; \widetilde{\delta \bp}, \; \widetilde{\delta \blambda}) \in \cZ^\infty$
	solves the adjoint equations
	 \begin{subequations}   \label{eq:Sensitivity:OCP:ocp-solution-sensitivity-system-adjoint} 
	 \begin{align}
	        \widetilde{\delta \bx}'(t) &= \overline{\BA}[t] \widetilde{\delta \bx}(t) + \overline{\BB}[t] \widetilde{\delta \bu}(t) + \overline{\BC}[t] \widetilde{\delta \bp}, & \aall t \in (t_0, t_f), \\ 
		\widetilde{\delta \bx}(t_0) &= 0, \\ 
		\widetilde{\delta \blambda}'(t) &= -  \overline{\BA}[t]^T \widetilde{\delta \blambda}(t)  -\overline{\BH}_{xx}[t] \widetilde{\delta \bx}(t) - \overline{\BH}_{xu}[t] \widetilde{\delta \bu}(t)  \nonumber \\
		                                                     &\quad- \overline{\BH}_{xp}[t] \widetilde{\delta \bp}
		                                       -  \nabla_x \overline{\ell}[t] - \overline{\bg}_x[t]^T \nabla_g \overline{\ell}[t]   & \aall t \in (t_0, t_f), \\ 
		 \widetilde{\delta \blambda}(t_f)  &= \nabla_x \overline{\phi}[t_f] +  \nabla_{xx}^2 \overline{\varphi}[t_f] \widetilde{\delta \bx}(t_f) + \nabla_{xp}^2 \overline{\varphi}[t_f] \widetilde{\delta \bp},
           \end{align}
           \begin{align}
		&\overline{\BH}_{ux}[t] \widetilde{\delta \bx}(t) + \overline{\BH}_{uu}[t] \widetilde{\delta \bu}(t) + \overline{\BH}_{up}[t] \widetilde{\delta \bp} + \overline{\BB}[t]^T \widetilde{\delta \blambda}(t) \nonumber \\
		&=  - \nabla_u \overline{\ell}[t] - \overline{\bg}_u[t]^T \nabla_g \overline{\ell}[t],  & \aall t \in (t_0, t_f), 
           \end{align}
           \begin{align}
	         & \nabla_{px}^2 \overline{\varphi}[t_f] \widetilde{\delta \bx}(t_f)  
	            +  \int_{t_0}^{t_f} \big( \overline{\BH}_{px}[t] \widetilde{\delta \bx}(t) + \overline{\BH}_{pu}[t] \widetilde{\delta \bu}(t) + \overline{\BC}[t]^T \widetilde{\delta \blambda}(t)  \big) \, dt \nonumber \\
	            & \quad +  \Big( \nabla_{pp}^2 \overline{\varphi}[t_f] +  \int_{t_0}^{t_f}  \overline{\BH}_{pp}[t]  \, dt \Big)  \widetilde{\delta \bp} = -  \nabla_p \overline{\phi}[t_f] -  \int_{t_0}^{t_f}    \big( \nabla_p \overline{\ell}[t] + \overline{\bg}_p[t]^T \nabla_g \overline{\ell}[t] \big)  \, dt.
         \end{align}
         \end{subequations}
         
         The conditions \eqref{eq:Sensitivity:OCP:ocp-solution-sensitivity-system-adjoint} are equivalent to the system of necessary and sufficient optimality conditions for the linear quadratic optimal control problem \eqref{eq:Sensitivity:OCP:LQOCP} with $(\delta \bx, \delta \bu, \delta \bp)$ replaced by $(\widetilde{\delta \bx}, \widetilde{\delta \bu}, \widetilde{\delta \bp})$ and using
          \begin{subequations}   \label{eq:Sensitivity:OCP:ocp-solution-sensitivity-system-adjoint-linear} 
	 \begin{align}
		\br(t) &\equiv 0,  &
		 \begin{pmatrix} \bc_x(t) \\ \bc_u(t) \end{pmatrix} 
		 &= \begin{pmatrix}   \nabla_x \overline{\ell}[t] + \overline{\bg}_x[t]^T \nabla_g \overline{\ell}[t] \\  
		                      \nabla_u \overline{\ell}[t] + \overline{\bg}_u[t]^T \nabla_g \overline{\ell}[t]
		        \end{pmatrix}, \\
		r_0 &= 0, \; \sigma_f = \nabla_x \overline{\phi}[t_f], &
		 \gamma &= \nabla_p \overline{\phi}[t_f] + \int_{t_0}^{t_f} \nabla_p \overline{\ell}[t] + \overline{\bg}_p[t]^T \nabla_g \overline{\ell}[t] \, dt
	 \end{align}
	  \end{subequations}  
	 with $\widetilde{\delta \blambda}$ as the costate.
\end{theorem}
\begin{proof}
   Equations (\ref{eq:Sensitivity:OCP:ocp-solution-sensitivity-system}a,b) and   (\ref{eq:Sensitivity:OCP:ocp-solution-sensitivity-system-adjoint}c,d) imply
   \begin{align}  \label{eq:Sensitivity:OCP:ocp-solution-sensitivity-system-adjoint1} 
        &  \nabla_x \overline{\varphi}[t_f]^T \delta \bx(t_f) + \delta \bx(t_f)^T \nabla_{xx}^2 \overline{\varphi}[t_f] \widetilde{\delta \bx}(t_f) + \delta \bx(t_f)^T \nabla_{xp}^2 \overline{\varphi}[t_f] \widetilde{\delta \bp}  \nonumber \\
        &= \widetilde{\delta \blambda}(t_f)^T  \delta \bx(t_f) - \widetilde{\delta \blambda}(t_0)^T \delta \bx(t_0)
            =  \int_{t_0}^{t_f} \widetilde{\delta \blambda}'(t)^T \delta \bx(t) + \widetilde{\delta \blambda}(t)^T \delta \bx'(t) \, dt \nonumber \\
        &= \int_{t_0}^{t_f} - \delta \bx(t)^T \overline{\BH}_{xx}[t] \widetilde{\delta \bx}(t) 
                                       -  \delta \bx(t)^T \overline{\BH}_{xu}[t] \widetilde{\delta \bu}(t) 
		                      -  \delta \bx(t)^T \overline{\BH}_{xp}[t] \widetilde{\delta \bp}   \, dt \nonumber \\
          &\quad  +  \int_{t_0}^{t_f} \widetilde{\delta \blambda}(t)^T \overline{\BB}[t] \delta \bu(t)  
                                          +  \widetilde{\delta \blambda}(t)^T \overline{\BC}[t] \delta \bp \, dt \nonumber \\
	  &\quad -  \int_{t_0}^{t_f}  \delta \bx(t)^T \big( \nabla_x \overline{\ell}[t] + \overline{\bg}_x[t]^T \nabla_g \overline{\ell}[t] \big)   \, dt
	                +  \int_{t_0}^{t_f} \widetilde{\delta \blambda}(t)^T  \overline{\bff}_{g}[t] \delta \bg[t]   \, dt.	                    
   \end{align}
    Equations (\ref{eq:Sensitivity:OCP:ocp-solution-sensitivity-system}c,d) and   (\ref{eq:Sensitivity:OCP:ocp-solution-sensitivity-system-adjoint}a,b) imply
   \begin{align}  \label{eq:Sensitivity:OCP:ocp-solution-sensitivity-system-adjoint2} 
        &  \delta \bx(t_f)^T  \nabla_{xx}^2 \overline{\varphi}[t_f] \widetilde{\delta \bx}(t_f) + \delta \bp^T \nabla_{px}^2 \overline{\varphi}[t_f] \widetilde{\delta \bx}(t_f) \nonumber \\
         &= \delta \blambda(t_f)^T  \widetilde{\delta \bx}(t_f) - \delta \blambda(t_0)^T \widetilde{\delta \bx}(t_0)
            =  \int_{t_0}^{t_f} \delta \blambda'(t)^T \widetilde{\delta \bx}(t) + \delta \blambda(t)^T \widetilde{\delta \bx}'(t) \, dt \nonumber \\
         &= \int_{t_0}^{t_f} - \widetilde{\delta \bx}(t)^T \overline{\BH}_{xx}[t] \delta \bx(t) 
                                       -  \widetilde{\delta \bx}(t)^T \overline{\BH}_{xu}[t] \delta \bu(t) 
		                      -  \widetilde{\delta \bx}(t)^T \overline{\BH}_{xp}[t] \delta \bp   \, dt \nonumber \\
           & \quad +  \int_{t_0}^{t_f} \delta \blambda(t)^T \overline{\BB}[t] \widetilde{\delta \bu}(t)  
                                          +  \delta \blambda(t)^T \overline{\BC}[t] \widetilde{\delta \bp} \, dt \nonumber \\
	   &\quad  -  \int_{t_0}^{t_f}  \widetilde{\delta \bx}(t)^T \big(  \overline{\BH}_{xg}[t] \delta \bg[t] + \delta \bg_x[t]^T \overline{\bd}[t] \big) \, dt.	         
   \end{align}
     Equations (\ref{eq:Sensitivity:OCP:ocp-solution-sensitivity-system}e) and   (\ref{eq:Sensitivity:OCP:ocp-solution-sensitivity-system-adjoint}e) imply
   \begin{align}  \label{eq:Sensitivity:OCP:ocp-solution-sensitivity-system-adjoint3} 
        & \int_{t_0}^{t_f} \widetilde{\delta \bu}(t)^T \big(  \overline{\BH}_{ux}[t] \delta \bx(t) + \overline{\BH}_{uu}[t] \delta \bu(t) + \overline{\BH}_{up}[t] \delta \bp + \overline{\BB}[t]^T \delta \blambda(t)  \big) \, dt    \nonumber \\
	&=  - \int_{t_0}^{t_f} \widetilde{\delta \bu}(t)^T \big(  \overline{\BH}_{ug}[t] \delta \bg[t] + \delta \bg_u[t]^T \overline{\bd}[t] \big) \, dt,      
   \end{align}
   \begin{align}  \label{eq:Sensitivity:OCP:ocp-solution-sensitivity-system-adjoint4} 
        & \int_{t_0}^{t_f} \delta \bu(t)^T \big(  \overline{\BH}_{ux}[t] \widetilde{\delta \bx}(t) + \overline{\BH}_{uu}[t] \widetilde{\delta \bu}(t) + \overline{\BH}_{up}[t] \widetilde{\delta \bp} + \overline{\BB}[t]^T \widetilde{\delta \blambda}(t)  \big) \, dt    \nonumber \\
	&=  - \int_{t_0}^{t_f} \delta \bu(t)^T \big(  \nabla_u \overline{\ell}[t] + \overline{\bg}_u[t]^T \nabla_g \overline{\ell}[t] \big) \, dt.
   \end{align}
      Equations (\ref{eq:Sensitivity:OCP:ocp-solution-sensitivity-system}f) and   (\ref{eq:Sensitivity:OCP:ocp-solution-sensitivity-system-adjoint}f) imply
   \begin{align}  \label{eq:Sensitivity:OCP:ocp-solution-sensitivity-system-adjoint5} 
              & \widetilde{\delta \bp}^T \nabla_{px}^2 \overline{\varphi}[t_f]\delta \bx(t_f)  
	            + \widetilde{\delta \bp}^T  \int_{t_0}^{t_f} \big( \overline{\BH}_{px}[t] \delta \bx(t) + \overline{\BH}_{pu}[t] \delta \bu(t)  \big) \, dt 
	          \nonumber \\ &\quad +  \widetilde{\delta \bp}^T \big( \nabla_{pp}^2 \overline{\varphi}[t_f] +  \int_{t_0}^{t_f}  \overline{\BH}_{pp}[t]  \, dt \big)  \delta \bp \nonumber \\
	        &= - \widetilde{\delta \bp} ^T \int_{t_0}^{t_f}   \overline{\BH}_{pg}[t] \delta \bg[t] + \delta \bg_p[t]^T \overline{\bd}[t] \, dt,
   \end{align}
   \begin{align}  \label{eq:Sensitivity:OCP:ocp-solution-sensitivity-system-adjoint6} 
          & \delta \bp^T \nabla_{px}^2 \overline{\varphi}[t_f] \widetilde{\delta \bx}(t_f) + \delta \bp^T \int_{t_0}^{t_f} \big( \overline{\BH}_{px}[t] \widetilde{\delta \bx}(t) + \overline{\BH}_{pu}[t] \widetilde{\delta \bu}(t)  \big) \, dt \nonumber \\
	        &\quad +  \delta \bp^T \Big( \nabla_{pp}^2 \overline{\varphi}[t_f] +  \int_{t_0}^{t_f}  \overline{\BH}_{pp}[t]  \, dt \Big)  \widetilde{\delta \bp} 
	  = -\delta \bp^T \nabla_p \overline{\phi}[t_f]  -  \delta \bp^T \big( \nabla_p \overline{\ell}[t] + \overline{\bg}_p[t]^T \nabla_g \overline{\ell}[t] \big).
   \end{align}
   Next, by adding/subtracting equations \eqref{eq:Sensitivity:OCP:ocp-solution-sensitivity-system-adjoint1} through \eqref{eq:Sensitivity:OCP:ocp-solution-sensitivity-system-adjoint6}, specifically, 
    taking \eqref{eq:Sensitivity:OCP:ocp-solution-sensitivity-system-adjoint1} -  \eqref{eq:Sensitivity:OCP:ocp-solution-sensitivity-system-adjoint2} +  \eqref{eq:Sensitivity:OCP:ocp-solution-sensitivity-system-adjoint3} - \eqref{eq:Sensitivity:OCP:ocp-solution-sensitivity-system-adjoint4}
    +  \eqref{eq:Sensitivity:OCP:ocp-solution-sensitivity-system-adjoint5} - \eqref{eq:Sensitivity:OCP:ocp-solution-sensitivity-system-adjoint6},
    we obtain
     \begin{align}  \label{eq:Sensitivity:OCP:ocp-solution-sensitivity-system-adjoint7} 
        &  \nabla_x \overline{\phi}[t_f]^T \delta \bx(t_f) +  \nabla_p \overline{\phi}[t_f]^T \delta \bp  
          +  \int_{t_0}^{t_f} \delta \bx(t)^T \big( \nabla_x \overline{\ell}[t] + \overline{\bg}_x[t]^T \nabla_g \overline{\ell}[t] \big)  \, dt	   \nonumber \\
        & + \int_{t_0}^{t_f} \delta \bu(t)^T \big(  \nabla_u \overline{\ell}[t] + \overline{\bg}_u[t]^T \nabla_g \overline{\ell}[t] \big) \, dt   
           +\delta \bp^T \int_{t_0}^{t_f} \big( \nabla_p \overline{\ell}[t] + \overline{\bg}_p[t]^T \nabla_g \overline{\ell}[t] \big)  \, dt  \nonumber \\
        &=  \int_{t_0}^{t_f} \widetilde{\delta \blambda}(t)^T  \overline{\bff}_{g}[t] \delta \bg[t]   \, dt	  
                 +  \int_{t_0}^{t_f}  \widetilde{\delta \bx}(t)^T \big(  \overline{\BH}_{xg}[t] \delta \bg[t] + \delta \bg_x[t]^T \overline{\bd}[t] \big) \, dt \nonumber \\
        &\quad +       \int_{t_0}^{t_f} \widetilde{\delta \bu}(t)^T \big(  \overline{\BH}_{ug}[t] \delta \bg[t] +  \delta \bg_u[t]^T \overline{\bd}[t] \big) \, dt
                     +  \widetilde{\delta \bp} ^T \int_{t_0}^{t_f}   \overline{\BH}_{pg}[t] \delta \bg[t] +  \delta \bg_p[t]^T \overline{\bd}[t] \, dt.
   \end{align}
   Inserting \eqref{eq:Sensitivity:OCP:ocp-solution-sensitivity-system-adjoint7}  into \eqref{eq:Sensitivity:OCP:q-Frechet_deriv} gives \eqref{eq:Sensitivity:OCP:sensitivity_result_qoi_adjoint}.
\end{proof}
The advantage of the adjoint-based approach is that at the cost of one LQOCP solve, one can compute the QoI sensitivity 
for any $\delta \bg$ by simply applying the linear operator \eqref{eq:Sensitivity:OCP:sensitivity_result_qoi_adjoint} rather than 
solving a different LQOCP for every $\delta\bg$.
Except for the linear terms \eqref{eq:Sensitivity:OCP:ocp-solution-sensitivity-system-linear} and
\eqref{eq:Sensitivity:OCP:ocp-solution-sensitivity-system-adjoint-linear}, the LQOCPs corresponding to
the sensitivity equation  \eqref{eq:Sensitivity:OCP:ocp-solution-sensitivity-system}
and to the adjoint equation \eqref{eq:Sensitivity:OCP:ocp-solution-sensitivity-system-adjoint} are identical,
owing to the symmetry of the necessary optimality conditions.


\section{Sensitivity-Based Error Estimate for QoI} \label{sec:QoI_error}
In \cite{JRCangelosi_MHeinkenschloss_2024a}, we used the sensitivity of an ODE with respect to a component function
$\bg$ to derive an error indicator for the change in the solution or the change in a QoI depending on the solution.
These error indicators can be extended to the optimal control context, and we present the error indicator for
the change in a QoI here.
Specifically, let $\widetilde{q} : \big( \cG^3(I) \big)^{n_g} \rightarrow \real$ be the QoI \eqref{eq:intro:QoI}. 
Assume we want to compute $\widetilde{q}$  at the solution of \eqref{eq:intro:OCP} with $\bg = \bg_*$,
but instead of  $\bg_*$ we only have an approximation $\widehat{\bg}$ and therefore can only compute 
the solution of \eqref{eq:intro:OCP} with $\bg = \widehat{\bg}$. 
We may estimate the error in the QoI as
\begin{equation} \label{eq:Refinement:OCP:approx-QoI-error}
	|\widetilde{q}(\widehat{\bg}) - \widetilde{q}(\bg_*)| 
	\approx | \widetilde{q}_\bg(\widehat{\bg})(\widehat{\bg} - \bg_*) |.
\end{equation}
The sensitivity $\widetilde{q}_\bg(\widehat{\bg})$ applied to $\delta \bg = \widehat{\bg} - \bg_*$ can
be computed via the adjoint equation approach in \eqref{eq:Sensitivity:OCP:sensitivity_result_qoi_adjoint} with $\overline{\bg} = \widehat{\bg}$.
Unfortunately, we do not know $\bg_*$, so we do not know $\delta \bg = \widehat{\bg} - \bg_*$.
However, if we have bounds for the error $\widehat{\bg} - \bg_*$ and its derivatives, then we can use
\eqref{eq:Sensitivity:OCP:sensitivity_result_qoi_adjoint} to compute an upper bound for 
$ | \widetilde{q}_\bg(\widehat{\bg})(\widehat{\bg} - \bg_*) |$.

Let  $(\overline{\bx}, \overline{\bu}, \overline{\bp} ) \in  \big( W^{1, \infty}(I) \big)^{n_x} \times \big( L^\infty(I) \big)^{n_u} \times \real^{n_p}$
be the solution of \eqref{eq:intro:OCP} with $\bg = \widehat{\bg}$ and
suppose pointwise error bounds
\begin{subequations} \label{eq:Refinement:intro:error-indicator-derivatives}
\begin{align}
    	\big|  \widehat{\bg}\big(t, \overline{\bx}(t), \overline{\bu}(t), \overline{\bp} \big) 
	        -  \bg_*\big(t, \overline{\bx}(t), \overline{\bu}(t), \overline{\bp} \big) \big| 
	&  \leq \bepsilon(t), &  \aall t \in I, \\
	\left| \frac{\partial}{\partial z} \widehat{\bg}\big(t, \overline{\bx}(t), \overline{\bu}(t), \overline{\bp} \big) 
	        -  \frac{\partial}{\partial z} \bg_*\big(t, \overline{\bx}(t), \overline{\bu}(t), \overline{\bp} \big) \right| 
	&  \leq \bepsilon^z(t),  & \aall t \in I \mbox{ and } z \in \{ x, u, p \}
\end{align}
\end{subequations}
are available for the function and its first partial derivatives along the trajectory obtained from $\widehat{\bg}$.
Since $\bg$ and $\frac{\partial}{\partial z} \bg$ are vector/matrix-valued, \eqref{eq:Refinement:intro:error-indicator-derivatives}
is understood elementwise.

We can now try to find an upper bound for $ | \widetilde{q}_\bg(\widehat{\bg})(\widehat{\bg} - \bg_*) |$
by maximizing $ | \widetilde{q}_\bg(\widehat{\bg}) \delta \bg  |$ over all perturbations $\delta \bg$ that
satisfy the bounds \eref{eq:Refinement:intro:error-indicator-derivatives}.
Using \eqref{eq:Sensitivity:OCP:sensitivity_result_qoi_adjoint} with $\overline{\bg} = \widehat{\bg}$ to compute $\widetilde{q}_\bg(\widehat{\bg}) \delta \bg$,
this leads to the optimization problem
\begin{subequations} \label{eq:Refinement:OCP:OCP-QoI}
\begin{align}
	\sup_{\delta \bg} \quad &\Big| \int_{t_0}^{t_f} \Big(  \overline{\BH}_{gx}[t]  \widetilde{\delta \bx}(t) 
                                                  + \overline{\BH}_{gu}[t]  \widetilde{\delta \bu}(t)
                                                  + \overline{\BH}_{gp}[t]  \widetilde{\delta \bp}
                                                  + \overline{\bff}_{g}[t]^T \widetilde{\delta \blambda}(t) 
                                                  +   \nabla_g \overline{\ell}[t]  \Big)^T  \delta \bg[t]   \, dt	\nonumber      \\ 
            &   \qquad +  \int_{t_0}^{t_f}  \Big(  \overline{\bd}[t]^T \delta \bg_x[t] \widetilde{\delta \bx}(t) \Big)  
                                             +  \Big(  \overline{\bd}[t]^T \delta \bg_u[t] \widetilde{\delta \bu}(t) \Big) 
                                             +  \Big(  \overline{\bd}[t]^T \delta \bg_p[t] \widetilde{\delta \bp} \Big)  \, dt \Big| \\
	\mbox{s.t.} \quad 
	& - \bepsilon(t) \leq \delta \bg\big(t, \overline{\bx}(t), \overline{\bu}(t), \overline{\bp} \big)  \leq \bepsilon(t), \qquad \aall t \in I, \\
	& - \bepsilon^x(t) \leq \delta \bg_x\big(t, \overline{\bx}(t), \overline{\bu}(t), \overline{\bp} \big)  \leq \bepsilon^x(t), \qquad \aall t \in I, \\
	& - \bepsilon^u(t) \leq \delta \bg_u\big(t, \overline{\bx}(t), \overline{\bu}(t), \overline{\bp} \big)  \leq \bepsilon^u(t), \qquad \aall t \in I, \\
	& - \bepsilon^p(t) \leq \delta \bg_p\big(t, \overline{\bx}(t), \overline{\bu}(t), \overline{\bp} \big)  \leq \bepsilon^p(t), \qquad \aall t \in I,
\end{align}
\end{subequations}
where the box constraints are understood elementwise.
Due to the presence of both $\delta \bg$ and its derivatives, it is unclear whether \eref{eq:Refinement:OCP:OCP-QoI} has a solution, let alone an easily computable one. Fortunately, this problem can be relaxed to obtain a problem with a simple analytical solution by replacing 
$\delta \bg\big(t, \overline{\bx}(t), \overline{\bu}(t), \overline{\bp} \big)$ and
$\delta \bg_x\big(t, \overline{\bx}(t), \overline{\bu}(t), \overline{\bp} \big)$,
$\delta \bg_u\big(t, \overline{\bx}(t), \overline{\bu}(t), \overline{\bp} \big)$,
$\delta \bg_p\big(t, \overline{\bx}(t), \overline{\bu}(t), \overline{\bp} \big)$
by
\[
	\bdelta \in \LL^{n_g}, \qquad
	\bdelta^x \in \LL^{n_g \times n_x}, \qquad 
	\bdelta^u \in \LL^{n_g \times n_u}, \qquad \
	\bdelta^p \in \LL^{n_g \times n_p}
\]
respectively to obtain 
\begin{subequations}   \label{eq:Refinement:OCP:OCP-QoI-delta}
\begin{align} 
	\max_{\bdelta, \bdelta^x, \bdelta^u, \bdelta^p} \quad &\Big| \int_{t_0}^{t_f} \!\! \Big(  \overline{\BH}_{gx}[t]  \widetilde{\delta \bx}(t) 
                                                  + \overline{\BH}_{gu}[t]  \widetilde{\delta \bu}(t)
                                                  + \overline{\BH}_{gp}[t]  \widetilde{\delta \bp}
                                                  + \overline{\bff}_{g}[t]^T \widetilde{\delta \blambda}(t) 
                                                  +   \nabla_g \overline{\ell}[t]  \Big)^T  \bdelta(t)   \, dt	    \nonumber \\ 
            &   \qquad +  \int_{t_0}^{t_f}  \Big(  \overline{\bd}[t]^T \bdelta^x(t) \widetilde{\delta \bx}(t) \Big)  
                                             +  \Big(  \overline{\bd}[t]^T \bdelta^u(t) \widetilde{\delta \bu}(t) \Big) 
                                             +  \Big(  \overline{\bd}[t]^T \bdelta^p(t) \widetilde{\delta \bp} \Big)  \, dt \Big|  \\ 
	\mbox{s.t.} \quad 
	& - \bepsilon(t) \leq \bdelta(t) \leq \bepsilon(t), \qquad \aall t \in I, \\
	& - \bepsilon^x(t) \leq \bdelta^x(t) \leq \bepsilon^x(t), \qquad \aall t \in I, \\
	& - \bepsilon^u(t) \leq \bdelta^u(t) \leq \bepsilon^u(t), \qquad \aall t \in I, \\
	& - \bepsilon^p(t) \leq \bdelta^p(t) \leq \bepsilon^p(t), \qquad \aall t \in I.
\end{align}
\end{subequations}

Note that $\bdelta^x$, $\bdelta^u$, $\bdelta^p$ in \eqref{eq:Refinement:OCP:OCP-QoI-delta}
are not derivatives of $\bdelta$; they are independent functions, which is why they are indicated by 
superscripts instead of subscripts. The same is true for $\bepsilon^x$, $\bepsilon^u$, $\bepsilon^p$. 

The optimization problem \eref{eq:Refinement:OCP:OCP-QoI-delta} is an infinite dimensional  linear program in 
the variables $\bdelta$,  $\bdelta^x$, $\bdelta^u$, $\bdelta^p$ and has  an analytical solution.

\begin{theorem} \label{thm:Refinement:OCP:OCP-acquisition-function}
     If  the assumptions of \cref{thm:Sensitivity:OCP:sensitivity_result_qoi-adjoint} hold
     and $\bepsilon$, $\bepsilon^x$, $\bepsilon^u$, $\bepsilon^p$ are essentially bounded, 
     then a solution of  \eqref{eq:Refinement:OCP:OCP-QoI-delta}  is given by 
     \begin{align*}
	\bdelta_i(t) &= \mathrm{sgn}\big(  \overline{\BH}_{gx}[t]  \widetilde{\delta \bx}(t) 
                                                  + \overline{\BH}_{gu}[t]  \widetilde{\delta \bu}(t)
                                                  + \overline{\BH}_{gp}[t]  \widetilde{\delta \bp}
                                                  + \overline{\bff}_{g}[t]^T \widetilde{\delta \blambda}(t) 
                                                  +   \nabla_g \overline{\ell}[t]  \big)_i  \bepsilon_i(t), \\
	\bdelta_{ij}^x(t) &= \mathrm{sgn}\big(\overline{\bd}_i[t] \widetilde{\delta \bx}_j(t)\big) \, \bepsilon_{ij}^x(t), \qquad  j = 1, \dots, n_x,  \\
	\bdelta_{ij}^u(t) &= \mathrm{sgn}\big(\overline{\bd}_i[t] \widetilde{\delta \bu}_j(t)\big) \, \bepsilon_{ij}^u(t), \qquad  j = 1, \dots, n_u,  \\
	\bdelta_{ij}^p(t) &= \mathrm{sgn}\big(\overline{\bd}_i[t] \widetilde{\delta \bp}_j\big) \, \bepsilon_{ij}^p(t), \quad\qquad j = 1, \dots, n_p 
    \end{align*}
    for $ i = 1, \dots, n_g$, with the corresponding objective value
    \begin{align*}
    	&\int_{t_0}^{t_f} \Big|  \overline{\BH}_{gx}[t]  \widetilde{\delta \bx}(t) 
                                                      + \overline{\BH}_{gu}[t]  \widetilde{\delta \bu}(t)
                                                      + \overline{\BH}_{gp}[t]  \widetilde{\delta \bp}
                                                      + \overline{\bff}_{g}[t]^T \widetilde{\delta \blambda}(t) 
                                                      +   \nabla_g \overline{\ell}[t]  \Big|^T  \bepsilon(t)   \, dt	    \nonumber \\ 
                &   \quad +  \int_{t_0}^{t_f}  \big| \overline{\bd}[t] \big|^T \bepsilon^x(t) \, \big| \widetilde{\delta \bx}(t) \big|
                                                 +  \big|  \overline{\bd}[t] \big|^T \bepsilon^u(t) \, \big| \widetilde{\delta \bu}(t) \big| +  \big|  \overline{\bd}[t] \big|^T \bepsilon^p(t) \, \big| \widetilde{\delta \bp} \big|  \, dt,
    \end{align*}
    where the absolute value is applied componentwise.
\end{theorem}
\begin{proof}
    The proof is analogous to the proof of   \cite[Thm.~3.6]{JRCangelosi_MHeinkenschloss_2024a}.
\end{proof}

Because of the error bound \eqref{eq:Refinement:intro:error-indicator-derivatives} and the fact that
 \eqref{eq:Refinement:OCP:OCP-QoI-delta} is a relaxation of \eqref{eq:Refinement:OCP:OCP-QoI},
\cref{thm:Refinement:OCP:OCP-acquisition-function} implies the following result.

\begin{theorem} \label{thm:Refinement:OCP:OCP-acquisition-function-bound}
     If  the assumptions of \cref{thm:Sensitivity:OCP:sensitivity_result_qoi-adjoint} hold,
     then the approximate QoI error measure \eqref{eq:Refinement:OCP:approx-QoI-error} satisfies the bound
	\begin{align*}
		&| \widetilde{q}_\bg(\overline{\bg})(\overline{\bg} - \bg_*) | \\ &\leq \int_{t_0}^{t_f} \Big|  \overline{\BH}_{gx}[t]  \widetilde{\delta \bx}(t) 
                                                 \! + \! \overline{\BH}_{gu}[t]  \widetilde{\delta \bu}(t)
                                                  \! + \! \overline{\BH}_{gp}[t]  \widetilde{\delta \bp}
                                                  \! + \! \overline{\bff}_{g}[t]^T \widetilde{\delta \blambda}(t) 
                                                  \! + \!   \nabla_g \overline{\ell}[t]  \Big|^T  \bepsilon(t)   \, dt	    \nonumber \\ 
            &   \quad +  \int_{t_0}^{t_f}  \big| \overline{\bd}[t] \big|^T \bepsilon^x(t) \, \big| \widetilde{\delta \bx}(t) \big|
                                             +  \big|  \overline{\bd}[t] \big|^T \bepsilon^u(t) \, \big| \widetilde{\delta \bu}(t) \big| +  \big|  \overline{\bd}[t] \big|^T \bepsilon^p(t) \, \big| \widetilde{\delta \bp} \big|  \, dt.
	\end{align*}
\end{theorem}


\section{Numerical Results} \label{sec:numerics}

The sensitivity results of \cref{thm:Sensitivity:OCP:OCP-solution-sensitivity} are used to predict the ability of a feedback controller 
to track a reference trajectory for a notional hypersonic vehicle in longitudinal flight. See \cref{fig:Refinement:ODE:dynamic_model} for a visual depiction of the dynamic model.
\begin{figure}[!htb]
\centering
	\begin{tikzpicture}[scale = 0.74]
	    \draw[thick, rotate=36, scale=1] (-2, -0.2) -- (0, 0) -- (-2, 0.2) -- cycle;
	    \draw[thick, dashed, color=olive, rotate=36, scale=2] (-1.3, 0) -- (1.8, 0) node[anchor=west]{chord line};
	    \draw[thick, dashed, color=red, rotate=0, scale=2] (0, 0) -- (1.75, 0) node[anchor=west]{horizon};
	    \draw[thick, color=blue, rotate=18, scale=2] (0, 0) -- (1.5, 0) node[anchor=west]{$v$};
	    \draw[thick, color=blue, rotate=18, scale=1] (0, 0) -- (-2, 0) node[anchor=east]{$D$};
	    \draw[thick, color=magenta, rotate=18, scale=1.2] (0, 0) -- (0, 1.6) node[anchor=south]{$L$};
	    \draw[thick, color=purple, rotate=18, scale=1.2] (0, 0.2) -- (-0.2, 0.2) -- (-0.2, 0);
	    \node[color=teal] at (1.3, 0.65) {$\alpha$};
	    \draw [thick, color=teal, domain=18:36] plot ({1.2 * cos(\x)}, {1.2 * sin(\x)});
	    \node[color=violet] at (1.3, 0.18) {$\gamma$};
	    \draw [thick, color=violet, domain=0:18] plot ({1 * cos(\x)}, {1 * sin(\x)});
	    \draw[scale=2] (-1.8, -1) -- (-1.8, 1) node[anchor=south]{$x_2$};
	    \draw[scale=2] (-1.8, -1) -- (2.5, -1) node[anchor=west]{$x_1$};
	    \draw[thick, color=magenta, rotate=36, scale=2] (-0.9, 0.0) -- (-1.3, 0.15);
	    \node[color=teal] at (-2.4, -1.5) {$\delta$};
	    \draw [thick, color=teal, domain=210:216] plot ({2.5 * cos(\x)}, {2.5 * sin(\x)});
	\end{tikzpicture}
\caption{Dynamic model for a hypersonic vehicle with control via flap deflection.} \label{fig:Refinement:ODE:dynamic_model}
\end{figure}
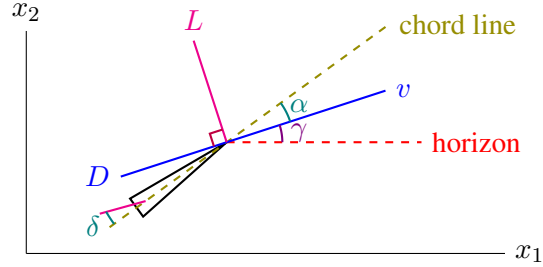
The states are downrange $x_1$ [m], altitude $x_2$  [m], speed $v$  [m/s], flight path angle $\gamma$  [rad], 
angle of attack $\alpha$ [rad],  and pitch rate $q$ [rad/s], i.e.,
in this example,
\[
       \bx(t) =  \big( \bx_1(t), \bx_2(t), \bv(t), \bgamma(t), \balpha(t), \bq(t) \big).
\]
The control surface is an elevator (or flap) that makes an angle $\delta$ [rad] with the chord line of the vehicle:
\[
	\bu(t) = \bdelta(t).
\]
Note that this $\bdelta$ is a physical quantity and is not related to the
$\bdelta$ in \eqref{eq:Refinement:OCP:OCP-QoI-delta}.

The flight duration $T$ [s] is also included as an optimization variable:
\[
	\bp = T.
\]

Formulas for derived quantities and other equations involving states are expressed in base units (kg/m/s) unless otherwise specified.
However, numerical results are reported in (kg/km/s), and using the (kg/km/s) system led to better numerical performance of the 
optimization algorithm. Angles appearing in formulas are in radians, but reported in degrees for figures.

The hypersonic vehicle considered in this example has mass $m = 1000 \textrm{ kg}$, moment of inertia $I_z = 247 \textrm{ kg $\times$ m$^2$}$ about the pitch axis, reference area $A_w = 4.4 \textrm{ m$^2$}$, and reference length $L_w = 3.6 \textrm{ m}$.

Lift, drag, and vehicle moment about the pitch axis are given by
   		\begin{align*}
			L(x_2, v, \alpha, \delta) &= \overline{q}(x_2, v) C_L(\alpha, \delta) A_w  &\textrm{[N]}, \\
			D(x_2, v, \alpha, \delta) &= \overline{q}(x_2, v) C_D(\alpha, \delta) A_w &\textrm{[N]}, \\
			M(x_2, v, \alpha, \delta) &= \overline{q}(x_2, v) C_M(\alpha, \delta) A_w L_w &\textrm{[N $\times$ m]}
		\end{align*}
where
\begin{equation} \label{eq:Refinement:ODE:DP}
	\overline{q}(x_2, v) = \frac{1}{2} \rho(x_2) v^2 \qquad \textrm{[Pa]}
\end{equation}
is the dynamic pressure, which depends on atmospheric density approximated by
\[
	\rho(x_2) = 1.225 \, \mathrm{exp}(-0.00014x_2) \qquad \textrm{[kg/m$^3$]}.
\]

The lift, drag, and moment coefficients $C_L, C_D, C_M$ for the vehicle are assumed to depend on angle of attack and flap deflection angle. They will play the role of the model function in this example, i.e.,
\[
      \bg\big(t, \bx(t), \bu(t), \bp\big) 
      =  \begin{pmatrix}  
            C_L\big(\balpha(t), \bdelta(t)\big)  \\  
            C_D\big(\balpha(t), \bdelta(t) \big) \\ 
            C_M\big(\balpha(t), \bdelta(t) \big) 
        \end{pmatrix}.
\]
In this example, the approximate lift, drag, and moment coefficients,
which compose $\widehat{\bg}$, are polynomial models given by
\begin{equation} \label{eq:Refinement:ODE:approximate-model}
\begin{aligned}
	\widehat{C_L}(\alpha, \delta) &= -0.04 + 0.8 \alpha + 0.13 \delta, \\
	\widehat{C_D}(\alpha, \delta) &= 0.012 - 0.01 \alpha + 0.6 \alpha^2 - 0.02 \delta + 0.12 \delta^2, \\
	\widehat{C_M}(\alpha, \delta) &= 0.1745 - \alpha - \delta.
	\end{aligned}
	\end{equation}
	and the true model $\bg_*$ is given by
\begin{equation} \label{eq:Refinement:ODE:true-model}
\begin{aligned}
	C_L^*(\alpha, \delta) = (1 + \epsilon) \widehat{C_L}(\alpha, \delta), &&
	C_D^*(\alpha, \delta) = (1 - \epsilon) \widehat{C_D}(\alpha, \delta), &&
	C_M^*(\alpha, \delta) = (1 + \epsilon) \widehat{C_M}(\alpha, \delta)
	\end{aligned}
	\end{equation}
	with $\epsilon = 0.05$. In later experiments we will vary $\epsilon$, changing the true model. The reason for this choice as opposed to varying the surrogate model is to ensure the reference trajectory, which is computed with $\widehat{\bg}$, remains the same. This choice also implies that $\widetilde{q}_\bg(\widehat{\bg}) (\bg_* - \widehat{\bg})$ varies linearly with $\epsilon$, as $\widehat{\bg}$ is fixed and $\bg_* - \widehat{\bg}$ depends linearly on $\epsilon$.

The dynamics of the hypersonic vehicle also depend on gravitational acceleration, computed as
\[
	g(x_2) = \mu \ / \ (R_E + x_2)^2
\]
where $\mu = 3.986 \times 10^{14} \textrm{ m$^3$/s$^2$}$ is the standard gravitational parameter and $R_E \approx 6.371 \times 10^6 \textrm{ m}$ is the radius of Earth.

The dynamics of the hypersonic vehicle are given by 
\begin{equation} \label{eq:Refinement:ODE:hypersonic_IVP}
\begin{aligned}
    \bx_1'(t) &= \bv(t) \cos \bgamma(t), \\
    \bx_2'(t) &= \bv(t) \sin \bgamma(t), \\
    \bv'(t) &= -\frac{1}{m} \Big( D\big(\bx_2(t), \bv(t), \balpha(t), \bdelta(t) \big) + m g\big(\bx_2(t)\big) \sin \bgamma(t) \big), \\
    \bgamma'(t) &= \frac{1}{m \bv(t)} \Big( L \big(\bx_2(t), \bv(t), \balpha(t), \bdelta(t) \big) - mg \big(\bx_2(t) \big) \cos \bgamma(t) + \frac{m \bv(t)^2 \cos \bgamma(t)}{R_E + \bx_2(t)} \Big), \\
    \balpha'(t) &= \bq(t) - \bgamma'(t) \\
    \bq'(t) &= M \big(\bx_2(t), \bv(t), \balpha(t), \bdelta(t) \big) \; / \; I_z, 
\end{aligned}
\end{equation}
for $\aall t \in (0, T)$, 
with initial conditions
\begin{equation} \label{eq:Refinement:ODE:initial-conditions}
 	\bx_1(0) = 0, \quad\! \bx_2(0) = 80000, \quad\! \bv(0) = 5000, \quad\! \bgamma(0) = -5 \pi / 180, \quad\!
	\balpha(0) = 11 \pi / 180, \quad\! \bq(0) = 0
\end{equation}
and box constraints
\begin{equation} \label{eq:numerics:box-constraints}
\begin{aligned}
 	0 \leq \bx_2(t) &\leq 81000, & 1 \leq \bv(t) &\leq 6000, &  t \in [0, T], \\
 	-30 \pi / 180 \leq \bgamma(t) &\leq 30 \pi / 180, &  0 \leq \balpha(t) &\leq 20 \pi / 180, &  t \in [0, T], \\
 	-20 \pi / 180 \leq \bdelta(t) &\leq 20 \pi / 180, &  \aall t \in (0, T), \qquad  1000 \leq T &\leq 3000.
\end{aligned}
\end{equation}

First, the maximum downrange problem
\begin{equation} \label{eq:numerics:maximum-downrange}
\begin{aligned}
	\max_{\bdelta, T} \quad & \bx_1(T) \\
	\mbox{s.t.} \quad & \mbox{ dynamics \eqref{eq:Refinement:ODE:hypersonic_IVP}}, 
	  \mbox{ initial conditions \eqref{eq:Refinement:ODE:initial-conditions}}, 
	  \mbox{ box constraints \eqref{eq:numerics:box-constraints}},
\end{aligned}
\end{equation}
was solved with $\bg = \widehat{\bg}$ to obtain a reference trajectory $(\overline{\bx}, \overline{\bu}, \overline{\bp})$. 
The problem has a variable duration $T$, so a fixed-time optimal control problem was obtained by 
converting the problem to normalized time, e.g., the first dynamic equation in \eqref{eq:Refinement:ODE:hypersonic_IVP}
 was replaced by
\begin{align*}
    \frac{d}{d\tau} \bx_1(T \tau) &= T \bv(T \tau) \cos \bgamma(T \tau), & \aall \tau \in (0, 1),
\end{align*}
with similar adjustments for the other constraints.

Given a reference trajectory $(\overline{\bx}, \overline{\bu}, \overline{\bp})$, we then consider the reference tracking problem
\begin{equation} \label{eq:Refinement:OCP:reference-tracking-OCP}
\begin{aligned}
	\min_{\delta, T} \quad & \int_0^1 \big( \bx(T \tau) - \overline{\bx}(\overline{T} \tau) \big)^T \BQ \big( \bx(T \tau) - \overline{\bx}(\overline{T} \tau) \big) \\ 
	&\quad + \big( \bu(T \tau) - \overline{\bu}(\overline{T} \tau) \big)^T \BR_u \big( \bu(T \tau) - \overline{\bu}(\overline{T} \tau) \big) \, dt 
	  + (\bp - \overline{\bp})^T \BR_p (\bp - \overline{\bp}) \\
	\mbox{s.t.} \quad & \mbox{dynamics \eqref{eq:Refinement:ODE:hypersonic_IVP}}
	                          \mbox{ and initial conditions \eqref{eq:Refinement:ODE:initial-conditions}}, 
\end{aligned}
\end{equation}
with weight matrices
\[
	\BQ := \mathrm{diag}(10^{-3}, 10^1, 0, 0, 10^1, 0), \qquad \BR_u := 10^8, \qquad \BR_p := 10^{-3}.
\]
This problem has the form \eqref{eq:intro:OCP} for sensitivity analysis, and by construction the optimal solution is the reference trajectory $(\bdelta, T) = (\overline{\bdelta}, \overline{T})$ when using the same model $\bg$ with which \eqref{eq:numerics:maximum-downrange} was solved, i.e., $\bg = \widehat{\bg}$. The optimal objective value of \eqref{eq:Refinement:OCP:reference-tracking-OCP} with $\bg = \bg_*$ represents the ability of a controller to track the reference trajectory computed with the approximation $\widehat{\bg}$ when the dynamical system is in fact subject to the true model $\bg_*$. The larger the objective value, the more difficult it is to track the reference trajectory, which could be an indication that $\widehat{\bg}$ is not a good enough approximation of $\bg_*$.

The sensitivity-based prediction of the solution error is given by
\[
	\delta \bx := \bx_\bg(\widehat{\bg}) (\bg_* - \widehat{\bg}) \approx \bx(\cdot \, ; \bg_*) - \bx(\cdot \, ; \widehat{\bg})
\]
where the sensitivity $\bx_\bg(\widehat{\bg}) \delta \bg$ is computed for $\delta \bg = \bg_* - \widehat{\bg}$ using \cref{thm:Sensitivity:OCP:OCP-solution-sensitivity}.
\begin{figure}[!htbp]
    \centering
    \includegraphics[width=0.48\textwidth]{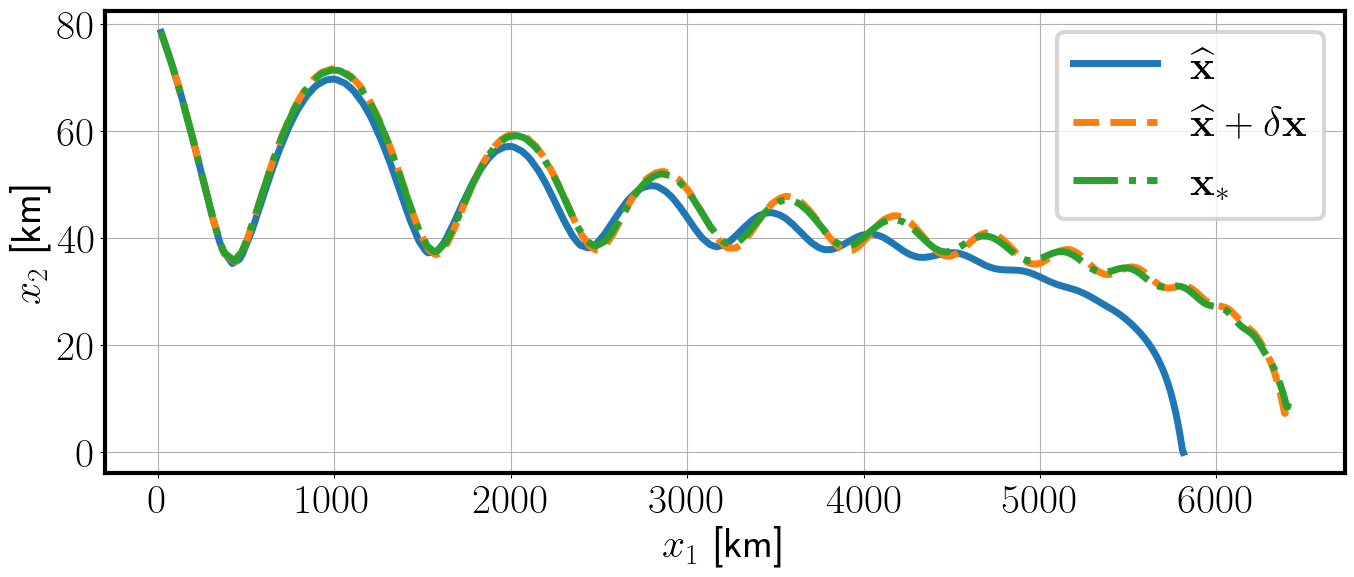}
    \includegraphics[width=0.48\textwidth]{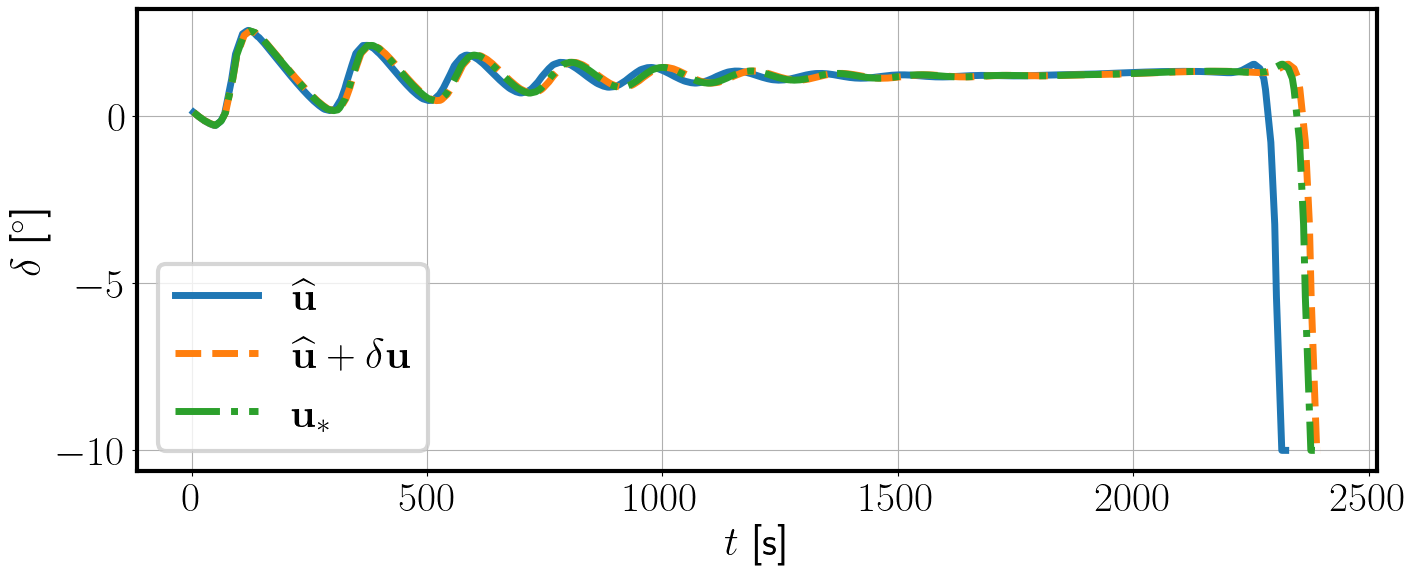}
    
    \vspace*{-1ex}
    \caption{Reference trajectory (blue), sensitivity-based prediction of optimal tracking trajectory (orange), and actual optimal tracking trajectory (green).}
    \label{fig:numerics:sensitivity-compare}
\end{figure}
\cref{fig:numerics:sensitivity-compare} shows three trajectories (just the states $x_1$ and $x_2$): namely, the reference trajectory $\bx(\cdot \, ; \widehat{\bg})$, the true optimal tracking trajectory $\bx(\cdot \, ; \bg_*)$, and $\bx(\cdot \, ; \widehat{\bg}) + \delta \bx$, which approximates $\bx(\cdot \, ; \bg_*)$. The plot shows that the sensitivity computations are reliable even for a fairly large $5\%$ relative error in the aerodynamic coefficient models. Figure~\ref{fig:numerics:sensitivity-compare} also shows similar results (with analogous notation) for the control input $\bu = \bdelta$. All optimal control problems were solved numerically using a flipped Legendre-Gauss-Radau collocation-based discretization of the dynamics and a compatible quadrature rule for the computation of integral terms in the objective function; see \cite{SKameswaran_LTBiegler_2008a}, \cite[Ch.~4]{JRCangelosi_2022a} for more details.

Next, several values of $\epsilon$ were considered in \eref{eq:Refinement:ODE:true-model}. As discussed earlier, we vary the true model $\bg_*$ instead of the approximation $\widehat{\bg}$ to ensure the reference trajectory remains the same. For each $\epsilon$, we solve \eref{eq:Refinement:OCP:reference-tracking-OCP} with the corresponding $\bg = \bg_*$ to examine the impact of $\epsilon$ on the downrange 
\[
	\widetilde{q}(\bg) := \bx_1(T; \bg).
\]
\begin{figure}[!htbp]
    \centering
    \includegraphics[width=0.50\textwidth]{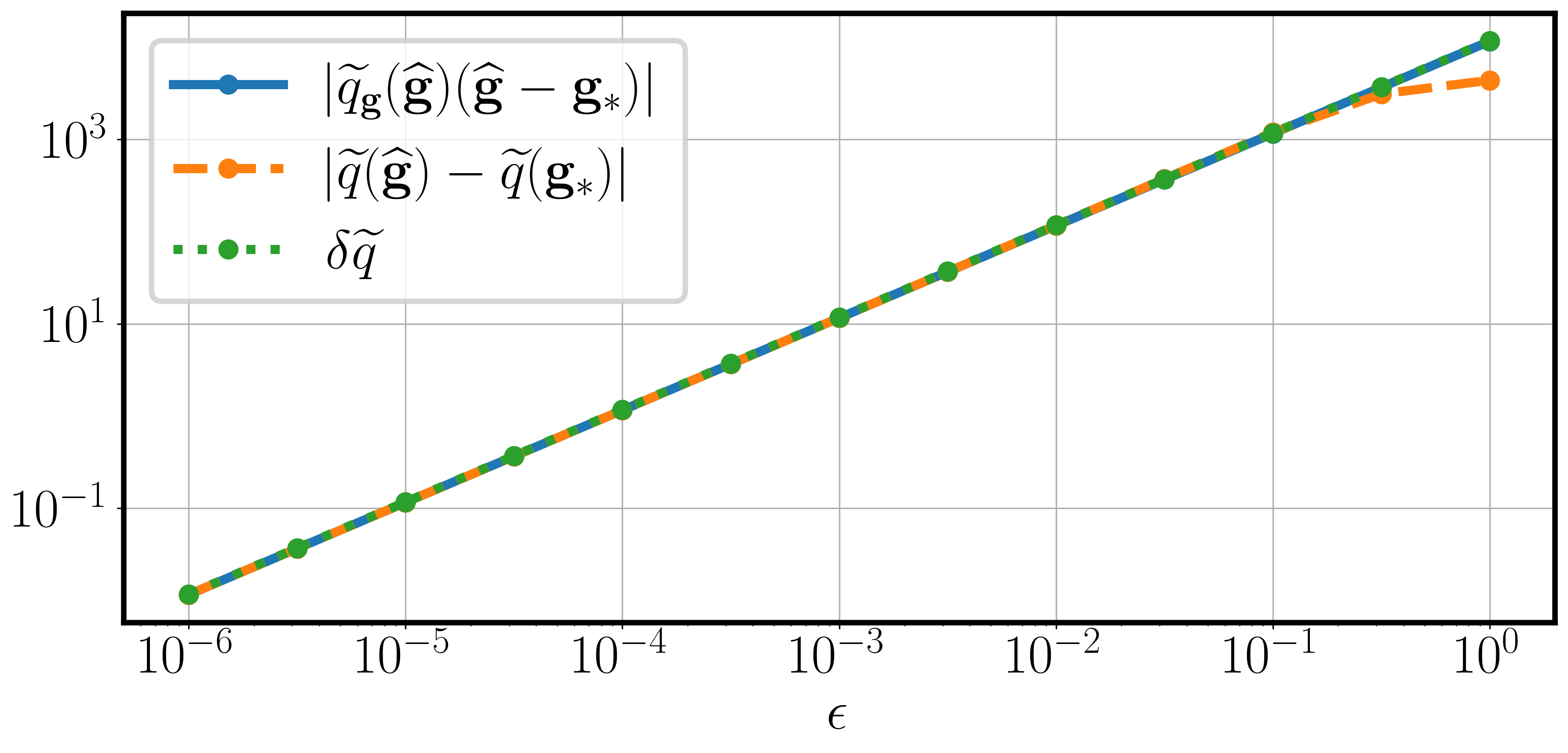}
    
    \vspace*{-1ex}
    \caption{Strong agreement between the QoI error estimates (blue), QoI errors (orange), and error bounds (green) for reference tracking OCP for a range of perturbations $\epsilon$.}
    \label{fig:numerics:QoI-study}
\end{figure}
\cref{fig:numerics:QoI-study} shows for each value of $\epsilon$ the sensitivity-based QoI error estimate $| \widetilde{q}(\widehat{\bg})(\widehat{\bg} - \bg_*) |$ of \cref{thm:Sensitivity:OCP:sensitivity_result_qoi-adjoint}, the true QoI error $| \widetilde{q}(\widehat{\bg}) - \widetilde{q}(\bg_*) |$, and the sensitivity-based approximate upper bound $\delta \widetilde{q}$ of \cref{thm:Refinement:OCP:OCP-acquisition-function-bound} using the pointwise error bounds
\begin{align*}
	\bepsilon(t) &= \big|  \widehat{\bg}\big(t, \overline{\bx}(t), \overline{\bu}(t), \overline{\bp} \big) 
	        -  \bg_*\big(t, \overline{\bx}(t), \overline{\bu}(t), \overline{\bp} \big) \big|, &  \aall t \in I, \\
	\bepsilon^z(t) &= \left| \frac{\partial}{\partial z} \widehat{\bg}\big(t, \overline{\bx}(t), \overline{\bu}(t), \overline{\bp} \big) 
	        -  \frac{\partial}{\partial z} \bg_*\big(t, \overline{\bx}(t), \overline{\bu}(t), \overline{\bp} \big) \right|,  & \aall t \in I \mbox{ and } z \in \{ x, u, p \},
\end{align*}
i.e., \eref{eq:Refinement:intro:error-indicator-derivatives} is satisfied with equality. The sensitivity-based QoI error estimate accurately predicts the actual QoI error for a wide range of $\epsilon$ values, and the sensitivity-based bound of \cref{thm:Refinement:OCP:OCP-acquisition-function-bound} is relatively tight. This gives some numerical evidence that the relaxation from \eref{eq:Refinement:OCP:OCP-QoI} to \eref{eq:Refinement:OCP:OCP-QoI-delta} does not meaningfully affect the supremum.


\section{Conclusions and Future Work} \label{sec:conclusion}

This paper has derived rigorous sensitivity analysis results for solutions of optimal control problems and related QoIs with respect to state- and control-dependent component functions
in the dynamics and objective function. These new sensitivity analyses were combined with pointwise estimates for the approximation error of the expensive-to-evaluate component functions to obtain a sensitivity-based approximate upper bound on the error in a QoI. This information can be used to assess the quality of the computed solution. In future work, these sensitivity-based estimates will be used to define an acquisition function that determines where to query the true/high-fidelity model to adaptively refine surrogate models for maximal improvement in solution quality.

%

\bibliographystyle{alphaurl}
\bibliography{references}

\end{document}